\documentclass[12pt, reqno]{amsart}
\usepackage{amsthm}
\usepackage{bm}
\usepackage[T1]{fontenc}
\usepackage{amsfonts}
\usepackage{amsmath}
\usepackage{amssymb}
\usepackage{mathrsfs}
\usepackage{epsfig}
\usepackage{graphicx}
\usepackage{color}

\makeatletter
\@namedef{subjclassname@2010}{%
\textup{2010} Mathematics Subject Classification}
\makeatother

\newtheorem{theorem}{Theorem}[section]
\newtheorem{corollary}[theorem]{Corollary}
\newtheorem{lemma}[theorem]{Lemma}

\theoremstyle{remark}
\newtheorem{remark}[theorem]{Remark}

\theoremstyle{definition}
\newtheorem{definition}[theorem]{Definition}
\newtheorem{example}[theorem]{Example}

\newcommand{\ran}{\mbox{ran}}

\newcommand{\bq}{\begin{equation}}
\newcommand{\eq}{\end{equation}}
\newcommand{\beqn}{\begin{eqnarray*}}
\newcommand{\eeqn}{\end{eqnarray*}}
\newcommand{\beq}{\begin{eqnarray}}
\newcommand{\eeq}{\end{eqnarray}}

\newcommand{\bc}{\begin{centre}}
\newcommand{\ec}{\end{centre}}

\newcommand{\ba}{\begin{array}}
\newcommand{\ea}{\end{array}}

\newcommand{\inp}[2]{\langle{#1},\,{#2} \rangle}

\renewcommand{\Delta}{{\nabla}}

\def \B{\mathcal{B}}

\def \D{\mathbb{D}}

\def \C{\mathbb{C}}
\def \F{\mathbb{F}}

\def \H{\mathcal{H}}
\def \S{\mathcal{S}}
\def \ZP{\mathbb{Z}_+}
\def \vp{\varphi}

\newcommand{\clb}{\mathcal{B}}

\newcommand{\cle}{\mathcal{E}}

\newcommand{\clh}{\mathcal{H}}
\newcommand{\clk}{\mathcal{K}}

\newcommand{\clm}{\mathcal{M}}

\newcommand{\cls}{\mathcal{S}}

\newcommand{\clw}{\mathcal{W}}

\newcommand{\n}{\bm{n}}
\newcommand{\m}{\bm{m}}
\newcommand{\z}{\bm{z}}
\newcommand{\w}{\bm{w}}
\newcommand{\BS}{\bm{S}}
\newcommand{\BR}{\bm{R}}
\newcommand{\BB}{\bm{B}}
\newcommand{\BW}{\bm{W}}
\newcommand{\s}{\bm{S}}
\newcommand{\M}{\bm{M}}
\newcommand{\J}{\bm{J}}
\newcommand{\raro}{\rightarrow}

\newcommand{\BP}{\bm{\Phi}}

\newcommand{\mf}{\mathfrak}

\newcommand*{\Le}{\leqslant}

\begin{document}

\title[Submodules in polydomains and noncommutative varieties]{Submodules in polydomains and noncommutative varieties}

\author[Das]{Susmita Das}
\address{Indian Statistical Institute, Statistics and Mathematics Unit, 8th Mile, Mysore Road, Bangalore, 560059,
India}
\email{susmita\_rs@isibang.ac.in}

\author[Pradhan]{Deepak Kumar Pradhan}
\address{Indian Statistical Institute, Statistics and Mathematics Unit, 8th Mile, Mysore Road, Bangalore, 560059,
India}
\email{deepak12pradhan@gmail.com}

\author[Sarkar]{Jaydeb Sarkar}
\address{Indian Statistical Institute, Statistics and Mathematics Unit, 8th Mile, Mysore Road, Bangalore, 560059,
India}
\email{jay@isibang.ac.in, jaydeb@gmail.com}



\begin{abstract}
Tensor product of Fock spaces is analogous to the Hardy space over the unit polydisc. This plays an important role in the development of noncommutative operator theory and function theory in the sense of noncommutative polydomains and noncommutative varieties. In this paper we study joint invariant subspaces of tensor product of full Fock spaces and noncommutative varieties. We also obtain, in particular, by using techniques of noncommutative varieties, a classification of joint invariant subspaces of $n$-fold tensor products of Drury-Arveson spaces.
\end{abstract}


\maketitle

\section{Introduction}

In this paper we study certain invariant subspaces (that is,
submodules) of tuples of operators (that is, Hilbert modules over
unital and associative free algebras generated by noncommuting
variables) in the setting of noncommutative operator theory and
noncommutative varieties. The noncommutative operator theory was
introduced in the middle eighties by Frazho \cite{F1, F2} and Bunce
\cite{B1} (see, however, Taylor \cite{JLT}). However, in the late
eighties, a more systematic formalism of noncommutative operator
theory began with the work of Popescu on isometric dilations and
analytic models of infinite sequences of noncommuting operators (see
\cite{Po89-1, Po89-2} and then \cite{Po06, G10, P18} and the
references therein). Popescu's noncommutative operator theory has a wide
range of applications in different context and research areas as,
free analysis and matrix convex sets \cite{PSS}, Hilbert
$C^*$-modules \cite{MS}, moment problem and Cuntz algebras \cite{DP} and operator algebras \cite{DK1, DK} and
multivariable operator theory and function theory \cite{Ar}, just to
name a few. On the other hand, the theory of noncommutative
varieties, again introduced by Popescu \cite{Po06}, establishes a
fundamental connection between noncommutative and commutative operator theory and function theory in several complex variables.

In the setting of noncommutative operator theory, noncommutative
polydomains and noncommutative varieties, introduced by Popescu in
\cite{Po13, Po15}, are analogue of polydisc in $\mathbb{C}^n$ .
Popescu's theory of polydomains can be seen as an attempt to unify
the function theory and multivariable operator theory (both
commutative and noncommutative) on the open unit ball and polydisc
like domains in $\C^n$.

The goal of the present paper is to examine a general technique for
characterizing joint invariant subspaces of the noncommutative Hardy
space on noncommutative polydomains and noncommutative varieties.

We emphasize that the notion of Fock space (we also call Fock
module) that plays the central role in noncommutative operator
theory and used in the free analytic models also plays significant
role in noncommutative polydomains. Here one actually needs to deal
with the tensor products of Fock spaces. From this point of view, in
this paper, we characterize invariant subspaces of tensor products
of Fock spaces. In order to be more specific, here we introduce the notions of Fock module, Fock $\n$-modules and multi-analytic maps, the most necessary technical background for the study of noncommutative multivariable operator theory, and refer the reader to Section \ref{section - Preliminaries} for a more detailed discussion.

Throughout this article, $n$ and $k$ will denote natural numbers and
$\bm{n} = (n_1, \ldots, n_k)$ will denote a $k$-tuple of natural
numbers. Consider the $n$-dimensional Hilbert space $\mathbb{C}^n$
with the standard orthonormal basis $\{e_1, \ldots, e_n\}$. The
\textit{Fock module} $F^2_n$ is defined by
\[
F^2_n = \bigoplus_{m\in \ZP} (\mathbb C^n)^{\otimes m},
\]
where $(\mathbb C^n)^{\otimes 0} = \mathbb C$ and $(\mathbb
C^n)^{\otimes m}$ is the $m$-fold Hilbert space tensor product of
$\C^n$. Define the {\it left creation operators} $S_1, \ldots, S_n$
on $F^2_n$ by $S_i f := e_i \otimes f$, $f \in F^2_n$. It is easy to
see that
\[
S^*_iS_j= \delta_{i,j}I_{F^2_n},
\]
for all $i, j = 1, \ldots, n$, that is, $(S_1, \ldots, S_n)$ is an
$n$ tuple noncommuting isometries with orthogonal ranges. Similarly,
we define the {\it right creation operators} $(R_1, \ldots, R_n)$ by
setting $R_i f = f \otimes e_i$, $f \in F^2_n$. \noindent The
\textit{Fock $\n$-module} $F^2_{\n}$ is defined by
\[
F^2_{\bm{n}} = F^2_{n_1} \otimes \cdots \otimes F^2_{n_k}.
\]
Now for each $i \in \{1, \ldots, k\}$, we denote the $n_i$-tuple of creation operators on $F^2_{n_i}$ (instead of $(S_1, \ldots, S_{n_i})$) by
\[
S_{n_i} = (S_{i 1}, \ldots, S_{i n_i}).
\]
Then, for each $j \in \{1, \ldots, n_i\}$ we define the operator $\BS_{ij}$ acting on the Fock $\n$-module $F^2_{\n}$ by setting
\[
\BS_{ij} := I_{F^2_{n_1}} \otimes \cdots \otimes I_{F^2_{n_{i-1}}} \otimes S_{i j} \otimes I_{F^2_{n_{i+1}}} \otimes \cdots I_{F^2_{n_k}}.
\]
It is now evident that $\s_{ij} \s_{pq} = \s_{pq} \s_{ij}$ and $\s_{ij}^* \s_{pq} = \s_{pq} \s_{ij}^*$ for all $1 \leq i < p \leq k$, $j = 1, \ldots, n_i$, and $q = 1, \ldots, n_p$. In other words, for each $i = 1, \ldots, k$,
\[
\BS_{n_i} :=(\BS_{i1}, \ldots, \BS_{i n_i}),
\]
is an $n_i$-tuple of noncommuting isometries with orthogonal ranges acting on the Fock $\n$-module $F^2_{\n}$. We set the $k$-tuple of tuples of noncommuting isometries $\BS$ as
\[
\BS = (\BS_{n_1}, \ldots, \BS_{n_k}).
\]
Finally, for Hilbert spaces $\cle$ and $\cle_*$, an operator $\Theta: F^2_n \otimes \cle \rightarrow F^2_n \otimes \cle_*$ is called \textit{multi-analytic} if
\[
\Theta (S_i \otimes I_{\cle}) = (S_i \otimes I_{\cle_*}) \Theta \quad \quad (i = 1, \ldots, n).
\]
The set of all such multi-analytic operators will be denoted by $R^{\infty}_{n} \overline \otimes \clb(\cle, \cle_*)$.

In this paper we classify joint invariant subspaces
of $\BS$. We also aim to illustrate our ideas in the setting of
noncommutative varieties (see Section \ref{section-variety}). For instance, one of our main results, Corollary \ref{second-BLHT}, states the following: Let $\clm$ be a closed subspace of the Fock $\n$-module $F^2_{\n}$. Let
\[
\cle = \clm \ominus \sum_{j=1}^{n_1} \BS_{1j} \clm,
\]
and $\cle_{\n} = F^2_{n_2} \otimes \cdots \otimes F^2_{n_k}$. Then $\clm$ is a submodule of $F^2_{\n}$ (that is, $X \clm \subseteq \clm$ for all $X \in \BS_{n_i}$, $i=1, \ldots, k$) if and only if there exist an isometric (also known as inner) multi-analytic operator $\Theta \in R^{\infty}_{n_1} \overline \otimes \clb(\cle, \cle_{\n})$ and multi-analytic operators $\Phi_{ij} \in R_{n_1}^{\infty} \overline \otimes \clb(\cle)$, for all $i = 2, \ldots, k$ and $j= 1, \ldots, n_i$, such that
\[
\mathcal M = \Theta (F^2_{n_1}\otimes \mathcal E) \quad \mbox{and} \quad \BS_{ij} \Theta = \Theta \Phi_{ij}.
\]
In this case, the multi-analytic operators $\Phi_{ij} \in R_{n_1}^{\infty} \overline \otimes \clb(\cle)$, $i = 2 , \ldots, k$ and $j = 1, \ldots, n_i$, are uniquely determined by $\clm$ (see Corollary \ref{second-BLHT} and Remark \ref{remark-13} for more details). Moreover, for each $i=2, \ldots, k$, the $n_i$-tuple $\BP_{n_i} = (\Phi_{i1}, \ldots, \Phi_{in_i})$ on $F^2_{n_1} \otimes \cle$ is a pure isometric tuple (see Remark \ref{remark-13}), whereas the restriction tuple
\[
(\BS_{n_1}|_{\clm}, \BS_{n_2}|_{\clm}, \ldots, \BS_{n_k}|_{\clm}) \; \mbox{on} \; \clm,
\]
and the pure isometric tuple
\[
(\BS_{n_1} \otimes I_{\cle}, \BP_{n_2}, \ldots, \BP_{n_k}) \; \mbox{on} \; F^2_{n_1} \otimes \cle,
\]
are jointly unitarily equivalent.

We then pass from Fock $\n$-module to constrained Fock $\n$-modules (a special class of quotient modules of the Fock $\n$-module) in the setting of noncommutative varieties and prove analogous results (for instance, see Corollary \ref{constrained-submodule}) for invariant subspaces of constrained tuples (compressions of the tuple of creation operators $\BS$ on $F^2_{\n}$ to constrained Fock $\n$-modules). In particular, in Corollary \ref{corollary-DA cor 1}, we present a classification of joint invariant subspaces \textit{Drury-Arveson $\n$-module} $H^2_{n_1} \otimes \cdots \otimes H^2_{n_k}$, the $k$-fold tensor products of Drury-Arveson spaces $\{H^2_{n_i}\}_{i=1}^k$.

At the present stage, we feel it is worthwhile to note that the problem of describing invariant subspaces of constrained tuples in the setting of noncommutative varieties is somewhat more challenging (essentially due to the non-uniqueness issue of the commutant lifting theorem, see Section \ref{section-DA modules}) and, on the other hand, our results are perhaps more definite for Drury-Arveson $\n$-module $H^2_{n_1} \otimes \cdots \otimes H^2_{n_k}$ (see Corollary \ref{corollary-DA cor 2}).

Part of the present investigation may be regarded as a generalization of some of the results concerning invariant subspaces of the Hardy module $H^2(\D^n)$ over unit polydisc $\D^n$ \cite{MMSS}. Besides, some of our ideas and arguments are inspired by the approach of \cite{MMSS}. However, we wish to point out that even in the particular case of $H^2(\D^n)$, our results are slightly different and also somewhat more convenient from that of \cite{MMSS} (see the final paragraph in Section \ref{section-submodules}).

In what follows, we will use the standard terminology of Hilbert modules. In
particular, this setting is more convenient and economic to deal
with the techniques involved in the results presented here.

This rest of the paper is organized as follows. Section \ref{section
- Preliminaries} contains preliminary notions, such as Fock module,
Hilbert modules over unital and associative free algebras,
multi-analytic and multi-coanalytic operators, and some basic
observations. Section \ref{section-commutators} contains a
convenient version of the proof of the noncommutative Beurling, Lax
and Halmos theorem and a relevant result concerning representations
of commutators of pure isometric tuples. Section
\ref{section-submodules} presents a classification result of joint
invariant subspaces of the Fock $\n$-module $F^2_{\bm{n}}$. Section
\ref{section-variety} discusses our approach of invariant subspaces
to noncommutative varieties and tensor product of noncommutative
varieties. Section \ref{section-DA modules} is devoted to the study
of submodules of Drury Arveson $\n$-modules. In particular, we
present explicit analytic descriptions of certain multi-analytic
maps associated with submodules of Drury-Arveson $\n$-modules.
Section \ref{concluding remarks} present an example of
dimension inequality of fibres of the noncommutative Beurling, Lax and Halmos theorem which also, in particular, show that certain natural generalizations of the classical results are not possible in noncommutative operator theory. This section also offers some concluding remarks.

\section{Preliminaries and basic observations}\label{section - Preliminaries}

Given two Hilbert spaces $\H_1$ and $\H_2$, the set of bounded
linear operators from $\H_1$ to $\H_2$ will be denoted by $\B(\H_1,
\H_2)$. If $\H_1 = \H_2$, then we shall write $\B(\H_1)$ for
$\B(\H_1, \H_1)$. Given an $n$-tuple $X = (X_1, \ldots, X_n)$ on
$\clh$, we define a completely positive map $Q_{X}: \mathcal B(\clh)
\rightarrow \mathcal B(\clh)$ (see \cite{Paul02}) by
\[
Q_{X}(A) := \sum_{i=1}^{n} X_i A X_i^* \quad \quad (A \in \clb(\mathcal H)).
\]
We say that $X$ is a \textit{row contraction} if $(I_{\clb(\clh)} - Q_X)(I_{\clh}) \geq 0$, or equivalently
\[
\sum_{i=1}^n X_i X_i^* \leq I_{\clh}.
\]
Therefore, if we denote
\[
\clb(\clh)^n = \{X = (X_1, \ldots, X_n): X_1, \ldots, X_n \in \clb(\clh)\},
\]
then, the set of all row contractions on $\clh$, given by
\[
{\mathfrak B}^{(n)}(\mathcal H): = \{X \in \clb(\H)^{n} :  (I_{\clb(\clh)} - Q_{X})(I_{\mathcal H}) \geq 0 \},
\]
is the noncommutative unit ball in $\clb(\H)^{n}$. Also it is easy
to see that if $X \in {\mathfrak B}^{(n)}(\mathcal H)$, then
\[
I_{\mathcal H} \geq Q_{X}(I_{\mathcal H}) \geq Q^{2}_{X}(I_{\mathcal H}) \geq \ldots \geq 0,
\]
which allows one to define a self-adjoint and contraction operator
$Q_X^{\infty}$ in $\clb(\clh)$ as
\[
Q^{\infty}_X := \mbox{SOT} -\lim_{l \rightarrow \infty} Q_{X}^l(I_{\mathcal H}).
\]
Note also that
\[
Q_{X}^m(I_{\mathcal H}) = \sum_{\overset{|\alpha| =
m}{\alpha \in F^+_{n}}}X^{\alpha} X^{*\alpha},
\]
where $F_n^+$ denotes the unital free semigroup on $n$ generators
$g_1, \ldots, g_n$ and the identity $e$, $X^{\alpha} = X_{i_1}
\cdots X_{i_m}$ and $|\alpha| = m$ for all $\alpha =  g_{i_1}\cdots
g_{i_m} \in F_n^+$. The following is now immediate:

\begin{lemma}\label{pure-criterion}
$Q^{\infty}_X=0$ if and only if $\mbox{SOT-}\displaystyle\lim_{m
\rightarrow \infty} \displaystyle \sum_{\overset{|\alpha| =
m}{\alpha \in F^+_{n}}}X^{\alpha} X^{*\alpha} = 0$.
\end{lemma}

The tuples of left and right creation operators are closely related to each other. In order to see this, we define
\[
e_{\alpha} =
\begin{cases} e_{i_1} \otimes \cdots \otimes e_{i_m} & \mbox{if}~ \alpha = g_{i_1}\cdots g_{i_m}
\\
\quad 1 & \mbox{if}~ \alpha = e, \end{cases}
\]
for all $\alpha \in F_n^+$. Next we define the \textit{flip operator} $U_t: F^2_n \rightarrow F^2_n$ by
\[
\quad U_t(e_\alpha) := e_{\alpha^t},
\]
where
\[
\alpha^t := g_{i_m}\cdots g_{i_1},
\]
and $\alpha = g_{i_1}\cdots g_{i_m} \in F^+_n$. Clearly, $U_t$
is unitary, $U_t^2 = I_{F^2_n}$ and
\[
U_t(f \otimes g) = U_t f \otimes U_t g,
\]
for all $f, g \in F^2_n$. Moreover, for $\alpha \in F^+_n$, since $R^\alpha f = f\otimes e_{\alpha^t}$, $f \in F^2_n$, it follows that
\[
R^{\alpha} = U_t S^{\alpha} U_t.
\]
Recall that the \textit{noncommutative
analytic Toeplitz algebra} $F^{\infty}_n$ (respectively
$R^{\infty}_{n}$) is the weakly closed algebra generated by the left
(respectively right) creation operators and the identity operator
$\{S_i: i=1, \ldots, n\} \cup \{I_{F^2_n}\}$ (respectively $\{R_i:
i=1, \ldots, n\} \cup \{I_{F^2_n}\}$).

Let $\C\langle Z_1,\ldots,Z_n\rangle$ denote the unital and
associative free algebra generated by $n$ noncommutative
variables $Z_1, \ldots, Z_n$ over $\C$. Then
\[
\C \langle Z_1,\ldots,Z_n\rangle = \bigoplus_{\alpha\in F^+_n} \C Z^{\alpha},
\]
where $Z^{\alpha}= Z_{i_1} \cdots Z_{i_m}$ for each word $\alpha =
g_{i_1}\ldots g_{i_m} \in F^+_n$. Now let $\{X_1, . . . , X_n\}$ be
(not necessarily commuting) bounded linear operators on a Hilbert
space $\clh$. We realize $\clh$ as a \textit{$\C\langle
Z_1,\ldots,Z_n\rangle$-Hilbert module} as follows:
\[
\C\langle Z_1,\ldots,Z_n\rangle \times \clh \rightarrow \clh,
\]
with
\[
(p(Z_1,\ldots, Z_n), f) \mapsto p(Z_1,\ldots, Z_n) \cdot f := p(X_1,\ldots, X_n)f,
\]
for $p(Z_1,\ldots, Z_n)$ in $\C\langle Z_1,\ldots,Z_n\rangle$ and $f \in \clh$. We say that $\mathcal H$ is a \textit{(left) Hilbert module}
corresponding to $X=(X_1,\ldots,X_n) \in \clb(\mathcal H)^n$. Often
we will simply say that $\clh$ is a $\C\langle
Z_1,\ldots,Z_n\rangle$-Hilbert module (or simply a Hilbert module
when no confusion can result) if $X$ is clear from the context.
Given $\C\langle Z_1,\ldots,Z_n\rangle$-Hilbert modules $\clh$ and
$\clk$ corresponding to $X \in \clb(\H)^n$ and $Y \in \clb(\clk)^n$,
respectively, a bounded linear operator $A \in \clb(\clh, \clk)$ is
said to be a module map if
\[
A X_i = Y_i A \quad \quad (i=1, \ldots, n).
\]
We say that $\clh$ is a \textit{row-contractive Hilbert Module} if
$(X_1, \ldots, X_n) \in \mf B^{(n)}(\mathcal H)$. In addition, if
$Q^{\infty}_X=0$, then we say that the Hilbert module $\clh$ is {\it
pure}.

Let $\clh$ be a row-contractive $\C \langle Z_1, \ldots, Z_n
\rangle$-Hilbert module, and let $\clm$ be a closed subspace of $\clh$. We say that $\clm$ is
a \textit{submodule} of $\clh$ if $X_i \clm \subseteq \clm$ for all
$i=1, \ldots, n$. In this case, we also treat $\clm$ as a $\C\langle Z_1,\ldots,Z_n\rangle$-Hilbert module corresponding to the $n$-tuple
\[
X|_{\clm} = (X_1|_{\clm}, \ldots, X_n|_{\clm}).
\]
We record for clarity and future use that the pure property of Hilbert modules carry over to submodules (see the first part of the proof of \cite[Theorem 3.2]{S-is}):

\begin{lemma}\label{pure-contractive-sub-module}
Any submodule of a pure and row-contractive $\C \langle Z_1,
\ldots, Z_n \rangle$-Hilbert module is pure and row-contractive $\C \langle Z_1,
\ldots, Z_n \rangle$-Hilbert module.
\end{lemma}
\begin{proof}
Let $\clh$ be a pure and row-contractive $\C \langle Z_1,
\ldots, Z_n \rangle$-Hilbert module, and let $\clm$ be a submodule of $\mathcal H$. For $h_1, \ldots, h_n \in \clm$, we have
\[
\|\sum_{i=1}^n X_i|_{\clm} h_i\|^2 = \|\sum_{i=1}^n X_i h_i\|^2 \leq \sum_{i=1}^n \|h_i\|^2,
\]
and hence $I_{\mathcal M}- Q_{X|_{\mathcal M}}(I_{\mathcal M}) \geq 0$ or, equivalently $X|_{\mathcal M} \in \mathfrak{B}^{(n)}(\mathcal M)$. Now, it can be checked easily, using
\[
X_i|_{\clm} (X_j|_{\clm})^* = (X_i P_{\clm} X_j^*)|_{\clm},
\]
for all $i,j = 1, \ldots, n$, that
\[
Q^m_{X|_{\mathcal M}}(I_{\mathcal M}) = \Big(\displaystyle
\sum_{\overset{|\alpha| = m}{\alpha \in F^+_{n}}}X^{\alpha} P_{\clm}
X^{*\alpha}\Big)|_{\clm},
\]
for all $m \geq 0$, and hence $Q^\infty_{X|_{\mathcal M}} = 0$.
\end{proof}

The quintessential example of pure and row-contractive Hilbert
modules over the noncommutative algebra $\C\langle
Z_1,\ldots,Z_n\rangle$ is the Fock module $F^2_n$ corresponding to
$(S_1, \ldots, S_n)$. If $n=1$, then the full Fock module $F^2_1$
can be identified with the {\it Hardy} module $H^2(\D)$ over the
unit disk and both $F^{\infty}_1, R^{\infty}_{1}$ coincide with
$H^{\infty}(\D)$. Also note that the $\C\langle
Z_1,\ldots,Z_n\rangle$-Hilbert module corresponding to the right
creation operators $(R_1, \ldots, R_n)$ is isometrically isomorphic,
via the flip operator, to the Fock module.

Given a Hilbert space $\cle$, the $\cle$-valued Fock module is the
$\C\langle Z_1,\ldots,Z_n\rangle$-Hilbert module $F^2_n \otimes
\mathcal{E}$ corresponding to the tuple
\[
S \otimes I_{\cle} = (S_1 \otimes I_{\cle}, \ldots, S_n \otimes   I_{\cle}) \in {\mathfrak B}^{(n)}(F^2_n \otimes \cle).
\]
When the Hilbert space $\cle$ is clear from the context, we also
write $S$ instead of $S \otimes I_{\cle}$. It is well known that

\[
\mathbb{C} \otimes \cle = \bigcap_{i=1}^{n} \ker (S_i \otimes I_{\cle})^*,
\]
and
\[
I_{F^2_n \otimes \cle}-\sum_{i=1}^{n} (S_i \otimes I_{\cle}) (S_i \otimes I_{\cle})^* = P_{\mathbb{C}} \otimes I_{\cle},
\]
where $P_{\mathbb{C}}$ denotes the orthogonal projection of $F^2_n$
onto the vacuum space $\mathbb{C} \subseteq F^2_n$. We say that a
bounded linear operator $\Theta: F^2_n \otimes \cle \rightarrow
F^2_n \otimes \cle_*$ is \textit{multi-analytic} if $\Theta$ is a
module map, that is
\[
\Theta (S_i \otimes I_{\cle}) = (S_i \otimes I_{\cle_*}) \Theta,
\]
for all $i=1, \ldots, n$. In this case, there exist unique
\textit{Fourier coefficients} $\theta_\alpha \in \clb(\cle, \cle_*)$,
$\alpha \in F_n^+$ \cite{Po95}, such that
\[
\Theta (R_1, \ldots, R_n) = \sum_{\alpha \in F^{+}_n} R^{\alpha} \otimes {\theta_{\alpha}}.
\]
Here the functional calculus of $\Theta (R_1, \ldots, R_n)$ is given by
\begin{equation}\label{multi-analytic}
\Theta (R_1, \ldots, R_n) = \mbox{{SOT-}}\lim_{r \rightarrow 1^{-}}
\sum_{m=0}^{\infty} \sum_{|\alpha| = m}
r^{|\alpha|}R^{\alpha}\otimes{\theta_{\alpha}}.
\end{equation}
Moreover, $\theta_\alpha \in \clb(\cle, \cle_*)$, $\alpha \in F_n^+$,
is uniquely determined by $\Theta$ as follows:
\begin{equation}\label{fourier-coeff}
\inp{\theta_{\alpha^t} \eta}{\zeta} = \inp{\Theta (1 \otimes
\eta)}{e_{\alpha}\otimes \zeta},
\end{equation}
for all $\eta \in \mathcal{E}$ and $\zeta \in \cle_*$. The above equality also shows that the vector $\Theta (1 \otimes \eta)$, $\eta \in \cle$, determines (and determined by) $\Theta$ uniquely. Hence we define $\theta : \cle \rightarrow F^2_n\otimes \cle_*$ by
\[
\theta \eta:= \Theta (1 \otimes \eta)  \quad \quad (\eta \in \mathcal E).
\]
It is well known (see for example \cite{ArPo95, DaPi99}) that the
set of all module maps from $F^2_n\otimes \mathcal E$ to
$F^2_{n}\otimes \mathcal E_*$ coincides with the weakly closed
subspace $R^{\infty}_n \overline \otimes \clb(\cle, \cle_*)$ generated
by the spatial tensor product $R^{\infty}_n \otimes_{sp} \clb(\cle,
\cle_*)$.

\textsf{Elements of $R^{\infty}_n \overline \otimes \clb(\cle, \cle_*)$
will be denoted by $\Theta(R_1, \ldots, R_n)$, or simply by $\Theta$
if $(R_1, \ldots, R_n)$ is clear from the context.}

A bounded linear operator $M \in \clb(F^2_n \otimes \cle, F^2_n \otimes \cle_*)$ is said to be \textit{multi-coanalytic} if
\[
M(S_i^* \otimes I_{\cle}) = (S_i^* \otimes I_{\cle_*})M,
\]
for all $i=1, \ldots, n$. If $\Theta \in \mathcal B{(F^2_n \otimes
\mathcal E, F^2_n \otimes \mathcal E_*)}$ is both multi-analytic and
multi-coanalytic, then by \eqref{fourier-coeff}, we have that
\[
\inp{\theta_{\alpha^t} \eta}{\zeta}  = \inp{(S^{\alpha*}\otimes
I_{\cle}) (1 \otimes \eta)}{\Theta^* (1\otimes \zeta)}=0,
\]
for all $\alpha \in F_n^+ \setminus \{e\}$, $\eta \in \cle$ and
$\zeta \in \cle_*$. On the other hand, if $\Theta \in \mathcal
B{(F^2_n \otimes \cle, F^2_n \otimes \cle_*)}$ and $\theta_\alpha =
0$ for all $F_n^+ \setminus \{e\}$, then one can easily check that
$\Theta$ is multi-coanalytic. Thus, we have proved the following
lemma:

\begin{lemma}\label{constant-fourier-coeff} A module map $\Theta \in \mathcal
B{(F^2_n \otimes \mathcal E, F^2_n \otimes \mathcal E_*)}$ is
multi-coanalytic if and only if the associated Fourier coefficients
$\theta_{\alpha} = 0$ for all $\alpha \in F_n^+ \setminus \{e\}$.
\end{lemma}

This also proves the following (well known) observation: A closed
subspace $\mathcal M \subseteq F^2_n\otimes \mathcal E$ is joint
reducing for $S \otimes I_{\cle}$ if and  only if there exists a
closed subspace $\clk \subseteq \cle$ such that $\mathcal M  = F^2_n
\otimes \mathcal K$. To prove the non-trivial implication, let
$\mathcal M$ is reducing for $S \otimes I_{\cle}$. Then the
orthogonal projection $P_{\mathcal M}$ onto $\mathcal M$ is a module
map. Since $P_{\mathcal M}$ is self-adjoint it is also a
multi-coanlaytic operator, and hence, by the above lemma, $P_{\clm}$
must be constant. Finally, since $P_{\mathcal M}$ is positive and
idempotent, it follows that
\[
P_{\clm} = I_{F^2_n} \otimes P_{\clk},
\]
for some $\clk \subseteq \cle$.

\section{Beurling, Lax and Halmos theorem and commutators}\label{section-commutators}

The classical Beurling, Lax and Halmos theorem \cite[page 198,
Theorem 3.3]{NF} deals with a complete classification of invariant
subspaces of vector-valued Hardy spaces over the open unit disc
$\D$. To be more specific, let $\cle_*$ be a Hilbert space, and let
$\cls$ be a closed subspace of $H^2_{\cle_*}(\D)$. Then the
Beurling, Lax and Halmos theorem says that $\cls$ is $M_z$-invariant
if and only if there exist a Hilbert space $\cle$ and an inner
function $\Theta \in H^\infty_{\clb(\cle, \cle_*)}(\D)$ such that
\[
\cls = \Theta H^2_{\cle}(\D).
\]
In particular, the shift $M_z$ on the Hardy space $H^2_{\cle}(\D)$ and the restriction operator $M_{z}|_{\S}$ on $\S$ are unitarily equivalent.
This has been generalized for Fock module by Popescu \cite{Po89-2}.
Here, however, we give a slightly direct (or geometric) approach to
reprove Popescu's result. This will be useful in characterizing
submodules of Fock $\n$-modules. Along the way we will also
parameterize commutators of tuples of noncommutative pure isometries.

Let $V = ( V_1,\ldots,V_n)\in \mathfrak B^n(\mathcal H)$ be a tuple
of isometries with orthogonal ranges. We call such a tuple a \textit{pure isometric tuple} if
\[
\mathcal Q_V^{\infty} =0.
\]
The following lemma will be also useful in analyzing commutators of pure isometric tuples.

\begin{lemma}\label{pure-smodules}
Let $\H$ be a $\mathbb{C}\langle Z_1, \ldots, Z_n \rangle$-Hilbert module corresponding to a pure isometric tuple $V=(V_1,\ldots,V_n)$.
Let
\[
\cle:= \H \ominus \displaystyle \sum_{i=1}^n V_{i} \H = \mathop {\bigcap}_{i=1}^n \ker V_i^*.
\]
Then
\[
\mbox{SOT}-\sum_{\alpha \in \F^+_n} V^\alpha P_{\mathcal{E}} V^{\alpha *} = I_{\mathcal H},
\]
and
\[
\H = \bigoplus_{\alpha \in F^+_n} V^{\alpha} \cle.
\]
Moreover, the map $L_V : \H \rightarrow F^2_n \otimes \cle$ defined by
\begin{equation}\label{eq-LV}
L_V ( V^{\alpha} \eta) := e_{\alpha} \otimes \eta = S^{\alpha} (1\otimes \eta) \quad \quad (\alpha \in F^+_n, \eta \in \cle),
\end{equation}
is a unitary module map and
\beqn
L_{V} f = \sum_{\alpha \in F^+_n} e_{\alpha}\otimes (P_{\mathcal E} V^{*\alpha} f) \quad  \quad (f\in \mathcal H).
\eeqn
\end{lemma}
\begin{proof}
We first note that $P_{\mathcal{E}} =  I_{\mathcal H} -\displaystyle\sum_{i = 1}^{n} V_iV_i^*$. Hence
\[
V^{\alpha} P_{\mathcal{E}} V^{\alpha *} = V^{\alpha} (I_H -\sum_{i = 1}^{n}  V_iV_i^*)V^{\alpha *} = V^{\alpha} V^{\alpha*} - \sum_{i=1}^{n} V^{\alpha}V_{i} V_{i}^* V^{\alpha*},
\]
for all $\alpha \in F^+_{n}$. For each $k \geq 1$, we have
\beqn
\sum_{|\alpha| \leq k} V^{\alpha} V^{\alpha*} = \sum_{l = 0}^{k} \sum_{|\alpha| = l }(V^{\alpha} V^{\alpha*}
- \sum_{|\alpha| = l+1} V^{\alpha} V^{\alpha*})
= I_{\mathcal H} - \sum_{|\alpha| = k+1} V^{\alpha} V^{\alpha*},
\eeqn
and hence the first equality follows from Lemma \ref{pure-criterion}. It is now easy to prove the second equality: Observe that $\displaystyle\bigoplus_{\alpha \in F^+_n} V^{\alpha} \cle$ is a joint reducing subspace of $V$. If $f \in \clh$ and 
\[
f \perp V^{\alpha} \cle,
\]
for all $\alpha \in F^+_n$, then  $V^{\alpha *} f \perp \cle$ and hence
\[
P_{\mathcal E} V^{*\alpha} f = 0.
\]
The first equality then implies that $f = 0$, which proves the validity of the second equality. The fact that $L_V$, as defined in \eqref{eq-LV}, is a unitary module map follows readily from the grading of $\clh$ in the second equality. The final equality follows from the first and the definition of $L_V$ in \eqref{eq-LV}.
\end{proof}

Given a pure isometric tuple $V = ( V_1,\ldots,V_n)\in \mathfrak B^n(\mathcal H)$, the unitary module map $L_V$ constructed in the above lemma is called the \textit{canonical module unitary operator} corresponding to $V$.

The following general fact will be useful: Given a closed subspace
$\cls$ of a Hilbert space $\clh$, the inclusion map $\iota_{\mathcal
S}: \cls \hookrightarrow \clh$ satisfies the following properties:
\[
\iota^*_{\cls} \iota_{\cls}  = I_{\cls}  \quad \mbox{ and }\quad
\iota_{\cls} \iota^*_{\cls} = P_{\cls}.
\]

Now, let $\cle_*$ be a Hilbert space, and let $\mathcal M \subseteq
F^2_n \otimes \cle_*$ be a submodule of $F^2_n \otimes \cle_*$.
Applying Lemma \ref{pure-smodules}, by virtue of Lemma
\ref{pure-contractive-sub-module}, to
\[
(S \otimes I_{\cle_*})|_{\clm} = ((S_i \otimes I_{\cle_*})|_{\clm},
\ldots, (S_n \otimes I_{\cle_*})|_{\clm}),
\]
we obtain the canonical module unitary operator corresponding to $(S \otimes I_{\cle_*})|_{\clm}$ as $L_{S|_{\clm}} : \clm \rightarrow F^2_n \otimes \cle$, where
\begin{equation}\label{eq-ker SI}
\cle = \mathop{\bigcap}_{i=1}^n \ker \Big((S_i \otimes I_{\cle_*})|_{\clm} \Big)^* \subseteq \clm.
\end{equation}
Then $\iota_{\mathcal M} L_{(S \otimes I_{\cle_*})|_{\clm}}^*: F^2_n \otimes \cle \raro F^2_n \otimes \cle_*$ is an isometric module map and
\[
\clm = \mbox{ran} (\iota_{\mathcal M} L_{(S \otimes I_{\cle_*})|_{\clm}}^*).
\]
If we define
\[
\Theta  = \iota_{\clm} L^*_{(S \otimes I_{\cle_*})|_{\clm}},
\]
then $\Theta \in R^{\infty}_n \overline{\otimes} \clb(\cle, \cle_*)$ is an \textit{inner} (that is, isometric) multi-analytic operator, and moreover
\[
\clm = \Theta  (F^2_n \otimes \mathcal E).
\]
We now proceed to the uniqueness part. Let $\tilde{\cle}$ be a Hilbert space, and let $\tilde{\Theta}$ be an inner multi-analytic operator in $R^{\infty}_n \overline{\otimes} \clb(\tilde{\cle}, \cle_*)$ such that $\clm = \tilde{\Theta} (F^2_n \otimes \tilde{\cle})$. Then for
\[
\Theta^* \tilde{\Theta}: F^2_n \otimes \tilde{\cle} \raro F^2_n \otimes \cle,
\]
we have
\[
(S_i \otimes I_{{\cle}})\Theta^* \tilde{\Theta} = L_V (S_i \otimes I_{{\cle}})|_{\clm} \iota_{\clm}^* \tilde{\Theta} =  L_V (S_i \otimes I_{{\cle}}) \tilde{\Theta} = L_V \tilde{\Theta} (S_i \otimes I_{{\cle}}),
\]
as $\mbox{ran} \tilde{\Theta} = \clm$ and $\iota_{\clm}^* \tilde{\Theta} = \tilde{\Theta}$, and hence
\[
(S_i \otimes I_{{\cle}})\Theta^* \tilde{\Theta} = \Theta^* \tilde{\Theta} (S_i \otimes I_{{\cle}}),
\]
for all $i = 1, \ldots, n$. Therefore $\Theta^* \tilde{\Theta}$ is a module map. On the other hand, since
\[
\Theta(F^2_n \otimes \cle) = \tilde{\Theta}(F^2_n \otimes \tilde{\cle}),
\]
for $h \in F^2_n \otimes \cle$, there exists $\tilde{h} \in F^2_n \otimes \tilde{\cle}$ such that $\Theta h = \tilde{\Theta} \tilde{h}$. Then we have
\[
(S_i \otimes I_{\tilde{\cle}}) \tilde{\Theta}^* \Theta h = (S_i \otimes I_{\tilde{\cle}}) \tilde{h} = \tilde{\Theta}^* (S_i \otimes I_{\cle_*}) \tilde{\Theta} \tilde{h} = \tilde{\Theta}^* (S_i \otimes I_{\cle_*}) \Theta h = \tilde{\Theta}^* \Theta (S_i \otimes I_{{\cle}}) h,
\]
for all $i=1, \ldots, n$, and hence that $\Theta^* \tilde{\Theta}$ is multi-coanalytic. Lemma \ref{constant-fourier-coeff} then implies that $\Theta^* \tilde{\Theta}$ is a constant map, that is, $\Theta^* \tilde{\Theta} = I_{F^2_n} \otimes \tau$ for some $\tau \in \clb(\tilde{\cle}, \cle)$. Thus
\[
\tilde{\Theta} = {\Theta} (I_{F^2_n} \otimes \tau),
\]
as $\Theta \Theta^* = P_{\mbox{ran} \tilde{\Theta}}$. That $\tau$ is a unitary follows from the fact that both $\Theta$ and $\tilde{\Theta}$ are isometries and $\Theta(F^2_n \otimes \cle) = \tilde{\Theta}(F^2_n \otimes \tilde{\cle})$.

Thus we have proved Popescu's noncommutative version of the classical Beurling-Lax-Halmos theorem (see \cite{ArPo95}, \cite{Po89-2} and \cite{DaPi99} for the original versions). 

\begin{theorem}\label{BLHT}
Let $\cle_*$ be a Hilbert space and let $\clm$ be a closed subspace of $F^2_n \otimes \cle_*$. Suppose
\[
\cle = \clm \ominus \displaystyle \sum_{i=1}^n (S_i \otimes I_{\cle_*}) \clm.
\]
Then the following are equivalent:

(i) $\clm$ is a submodule of $F^2_n \otimes \cle_*$.

(ii) There exist an inner multi-analytic operator $\Theta: F^2_n \otimes \cle \raro F^2_n \otimes \cle_*$ such that $\clm = \Theta(F^2_n \otimes \cle)$.

\noindent Moreover, in the latter case, if $\clm = \tilde{\Theta}(F^2_n \otimes \tilde{\cle})$ for some Hilbert space $\tilde{\cle}$ and inner multi-analytic operator $\tilde{\Theta}: F^2_n \otimes \tilde{\cle} \raro F^2_n \otimes \cle_*$, then there exists a unitary $\tau \in \clb(\tilde{\cle}, \cle)$ such that
\[
\tilde{\Theta} = {\Theta} (I_{F^2_n} \otimes \tau).
\]
\end{theorem}

It is worthwhile to note that our approach and presentation is slightly different than that of Popescu. Moreover, the present approach will be useful in the study of submodules of Fock module and Fock $\n$-modules.

We now turn to the representations of commutators of pure isometric tuples.

\begin{theorem}\label{comm-fourier-coeff}
Let $\H$ be a $\mathbb{C}\langle Z_1, \ldots, Z_n \rangle$-Hilbert module corresponding to a pure isometric tuple $V=(V_1,\ldots,V_n)$. Let $L_V: \clh \raro F^2_n \otimes \cle$ be the canonical module unitary operator. Then $C \in \{V_1,\ldots,V_n\}'$ if and only if there exists a multi-analytic operator $\Phi \in R^{\infty}_n \overline \otimes \clb(\cle)$ such that $L_V C L_V^* =\Phi$ and the Fourier coefficients of $\Phi$ are given by
\[
\vp_{\alpha^t} = P_{\cle} V^{\alpha*} C|_{\cle} \quad\quad (\alpha \in  F^{+}_{n}).
\]
\end{theorem}
\begin{proof}
Let $C \in \mathcal B(\mathcal H)$ and let $\eta \in \cle$. Then $C L_{V}^* (1 \otimes \eta) = C \eta$, as $L_V^*(1 \otimes \eta) = \eta$ by \eqref{eq-LV}. By the definition of $L_V$, we have
\[
L_{V}C L^*_{V} (1 \otimes \eta) = L_{V} C \eta = \sum_{\alpha\in F^+_n} e_{\alpha} \otimes (P_{\cle} V^{\alpha *} C \eta).
\]
Clearly, if $C\in \{V_1,\ldots,V_n\}'$, then
\[
(L_V C L^*_V) (S_i \otimes I_{\cle}) = (S_i \otimes I_{\cle}) (L_V C L^*_V),
\]
for all $i = 1, \ldots, n$, and hence $L_{V}C L^*_{V}$ is a multi-analytic operator. Let
\[
L_{V}C L^*_{V} = \Phi \in R^{\infty}_n \overline{\otimes} \clb(\cle).
\]
Then by \eqref{multi-analytic} we have
\beqn
\Phi(R_1, \ldots, R_n) = \mbox{SOT}-\lim_{r\rightarrow 1^-}\sum_{m=0}^{\infty} \sum_{|\alpha| = m} r^{|\alpha|} R^{\alpha}\otimes{\vp_{\alpha}},
\eeqn
where $\vp_{\alpha} \in \clb(\mathcal E)$ for all $\alpha \in F^+_n.$ Finally, if $\eta, \zeta \in \mathcal E$, then by \eqref{fourier-coeff} we have
\[
\begin{split}
\inp{\vp_{\alpha^t}\eta}{\zeta} & = \inp{ L_V C L_V^* (1\otimes \eta )}{e_{\alpha}\otimes \zeta}
\\
& = \inp{\sum_{\beta \in F^+_n} e_{\beta} \otimes (P_{\cle} V^{\beta *} C \eta)}{e_{\alpha}\otimes \zeta}
\\
& = \inp{P_{\cle} V^{\alpha *} C \eta}{\zeta},
\end{split}
\]
and hence $\vp_{\alpha^t} = P_{\cle} V^{\alpha *} C|_{\cle}$ for all $\alpha \in F^+_n$. The converse is obvious.
\end{proof}

Therefore, the unique formal Fourier expansion of $\Phi \in R^{\infty}_n \overline \otimes \clb(\cle)$ in the above proposition is given by
\[
\Phi(R_1, \ldots, R_n) = \sum_{\alpha \in F_n^+} R^{\alpha} \otimes \Big(P_{\cle} V^{\alpha^t *} C|_{\cle}\Big).
\]

Finally, a word of caution is necessary here: Since
\[
\cle = \bigcap_{i=1}^n \ker V_i^* \subseteq \clh,
\]
the representation of $\vp_{\alpha^t}$, $\alpha \in F_n^+$, in Theorem \ref{comm-fourier-coeff} is well-defined.

\section{Submodules of Fock $\n$-modules}\label{section-submodules}

This section deals with representations of submodules of the Fock $\n$-module $F^2_{\n}$. This problem originated from the natural question whether Beurling, Lax and Halmos type inner function based characterizations of invariant subspaces can be valid on Hardy space over unit polydisc $\D^n$, $n >1$. The answer is negative even for $n=2$ (see Rudin \cite[Theorems 4.1.1 and 4.4.2]{R}). On the other hand, recently in \cite{MMSS}, an abstract classification of invariant subspaces of the Hardy space over the unit polydisc has been proposed. Here we refine the method of \cite{MMSS} to handle the noncommutative tuples and carry out the classification of submodules of $F^2_{\n}$.

Let $\cle$ and $\clk$ be Hilbert spaces and let $T \in \clb(\clk)$. We treat $\clk$ as a $\C[Z]$-Hilbert module corresponding to $T$. Now consider the $\cle$-valued Fock module $F^2_n \otimes \cle$ as $\C\langle Z_1,\ldots,Z_n\rangle$-Hilbert module corresponding to $S \otimes I_{\cle} = (S_1 \otimes I_{\cle}, \ldots, S_n \otimes I_{\cle})$. Consider the free algebra
\[
\C\langle Z_1,\ldots,Z_n\rangle \otimes_{\C} \C[Z] = \C\langle Z_1,\ldots,Z_n, Z\rangle,
\]
generated by the indeterminates $\{Z_1, \ldots, Z_n, Z\}$. Note that
\[
(Z \otimes 1) (1 \otimes Z_i) - (1 \otimes Z_i) (Z \otimes 1) =0 \quad \quad (i =1, \ldots, n).
\]
Now we treat $(F^2_n \otimes \clk) \otimes \cle$ as a $\C\langle Z_1,\ldots,Z_n, Z\rangle$-Hilbert module, where $Z_i$ and $Z$ corresponds to $(S_i \otimes I_{\mathcal K}) \otimes I_{\cle}$ and $(I_{F^2_n}\otimes T) \otimes I_{\cle}$, respectively, and $1\Le i\Le n$. Note that since $(S_i \otimes I_{\mathcal K}) \otimes I_{\cle}$, $i = 1, \ldots, n$, commutes (and doubly commutes) with $(I_{F^2_n} \otimes T) \otimes I_{\cle}$, the above identification is well defined.

On the other hand, let $\clh$ be a $\C\langle Z_1,\ldots,Z_n\rangle$-Hilbert module corresponding to a pure isometric tuple $V\in \mathfrak B^{(n)}(\H)$ and let $T\in \mathcal B(\H)$. Suppose $T$ commutes and also doubly commutes with $V$. Let $F^2_n \otimes \cle$ be the identification of $\clh$ as in \eqref{eq-LV}. It follows from Lemma \ref{pure-smodules} that the representation of $T$ in $F^2_n \otimes \cle$ is a constant multi-analytic operator.

The above discussion is the underlying theme of this section, where we aim at characterizing joint invariant subspaces of Fock $\n$-module $F^2_{\n}$. Such a characterization is a consequence of the following key theorem.

\begin{theorem}\label{first-BLHT}
Let $\cle_*$ be a Hilbert space, $T \in \clb(F^2_n \otimes \cle_*)$ and let $\clm$ a closed subspace of $F^2_n \otimes \cle_*$. Let
\[
\cle = \clm \ominus \sum_{i=1}^n(S_i \otimes I_{\cle_*}) \clm.
\]
and suppose $T (S_i \otimes I_{\cle_*}) = (S_i \otimes I_{\cle_*}) T$ for all $i = 1, \ldots, n$. Then the following statements are equivalent:

\begin{itemize}
\item[(a)] $\mathcal M$ is submodule of $F^2_n \otimes \cle_*$ and $T\clm \subseteq \clm$.

\item[(b)] There exist an inner multi-analytic operator $\Theta \in R_n^{\infty}\overline{\otimes} \clb(\cle, \cle_*)$ and a multi-analytic operator $\Phi \in R^{\infty}_n\overline{\otimes} \mathcal B(\cle)$ such that
\[
\clm  =\Theta(F^2_n \otimes \cle),
\]
and
\[
T \Theta = \Theta \Phi.
\]
\end{itemize}
Moreover, if either of the above conditions hold, then the Fourier coefficients of $\Phi$ are given by
\[
\vp_{\alpha^{t}} = P_{ \cle} (S^{\alpha*}\otimes I_{\cle_*})T|_{\cle} \quad \quad (\alpha \in F^+_n).
\]
\end{theorem}
\begin{proof}
The implication $(b) \Rightarrow (a)$ follows from the well known Douglas' range-inclusion theorem. So we proceed to prove that $(a) \Rightarrow (b)$. Suppose $\clm$ is a submodule of $F^2_n \otimes \cle_*$ and suppose that $\clm$ is $T$-invariant. Define $\tilde{T}= T|_{\clm}$ and
\[
(S \otimes I_{\cle_*})|_{\clm} = ((S_1 \otimes I_{\cle_*})|_{\clm}, \ldots, (S_n \otimes I_{\cle_*})|_{\clm}).
\]
Clearly $(S_i \otimes I_{\cle_*})|_{\clm}$ is an isometry and $[(S_i \otimes I_{\cle_*})|_{\clm}, \tilde{T}] = 0$ for all $i=1, \ldots, n$. Hence, taking into account of Lemma \ref{pure-contractive-sub-module}, it follows that $(S \otimes I_{\cle_*})|_{\clm}$ is a pure isometric tuple. Now we are in the setting of the proof of Theorem \ref{BLHT}. Therefore
\[
\clm = \Theta  (F^2_n \otimes \mathcal E),
\]
where
\[
\cle = \mathop{\bigcap}_{i=1}^n \ker \Big((S_i \otimes I_{\cle_*})|_{\clm}\Big)^*,
\]
and
\[
\Theta  = \iota_{\clm} L^*_{(S \otimes I_{\cle_*})|_{\clm}} \in R^{\infty}_n \overline{\otimes} \clb(\cle ,\cle_*).
\]
Moreover, by Theorem \ref{comm-fourier-coeff}, there exists $\Phi \in R^{\infty}_{n} \overline{\otimes} \clb(\cle)$ such that
\[
\Phi= L_{(S \otimes I_{\cle_*})|_{\clm}} \tilde T L_{(S \otimes I_{\cle_*})|_{\clm}}^*,
\]
and the Fourier coefficients of $\Phi$ are given by
\[
\vp_{\alpha^{t}} = P_{\cle} \Big((S \otimes I_{\cle_*})|_{\clm}\Big)^{\alpha*} \tilde T|_{\cle} \quad \quad (\alpha \in F_n^+).
\]
Now for each $\alpha \in F_n^+$ we have
\[
P_{\cle} \Big((S \otimes I_{\cle_*})|_{\clm}\Big)^{\alpha*} \tilde T|_{\cle} = P_{\cle} P_{\clm} (S \otimes I_{\cle_*})^{\alpha*}|_{\clm} \tilde T|_{\cle} = P_{\cle} P_{\clm} (S^{\alpha*} \otimes I_{\cle_*}) \tilde T|_{\cle},
\]
hence, by the fact that $\cle \subseteq \clm$, we have 
\[
P_{\cle} \Big((S \otimes I_{\cle_*})|_{\clm}\Big)^{\alpha*} \tilde T|_{\cle}  = P_{\cle} (S^{\alpha*} \otimes I_{\cle_*}) \tilde T|_{\cle}.
\]
Finally, from the definitions of $\Theta$ and $\Phi$ above and the fact that $\tilde{T} = T|_{\clm} = \iota^*_{\clm} T \iota_{\clm}$, we conclude that
\[
\Phi= L_{(S \otimes I_{\cle_*})|_{\clm}} \tilde T L_{(S \otimes I_{\cle_*})|_{\clm}}^* = L_{(S \otimes I_{\cle_*})|_{\clm}} \iota^*_{\clm} T \iota_{\clm} L_{(S \otimes I_{\cle_*})|_{\clm}}^* = \Theta^* T \Theta,
\]
that is
\begin{equation}\label{eq-thetaphithete}
\Phi = \Theta^* T \Theta,
\end{equation}
and hence
\[
\Theta  \Phi= T \Theta,
\]
as $\Theta \Theta^* =P_{\mathcal M}$ and $\mbox{ran} (T \Theta) \subseteq \clm$. This completes the proof of the theorem.
\end{proof}

\begin{remark} \label{phi-is-shift}
In the setting of Theorem \ref{first-BLHT}, if, in addition, $T$ is an isometry, then
\[
\|\Phi h\| = \|\Theta \Phi h\| = \|T \Theta h\| = \|h\| \quad \quad (h \in F^2_n \otimes \cle),
\]
and hence it follows that $\Phi$ is also an isometry. Suppose now that $T$ is pure, that is, $T^{*m} \raro 0$ as $m \raro \infty$ in the strong operator topology. Then
\[
\Phi^{*m} = \Theta^* T^{*m} \Theta,
\]
for all $m \geq 1$ implies that $\Phi$ is also pure.
\end{remark}

Now we proceed to submodules of the Fock $\n$-module $F^2_{\n}$. Recall that the Fock $\n$-module $F^2_{\n}$, for $\n = (n_1, \ldots, n_k) \in \mathbb{N}^k$,  is given by
\[
F^2_{\n} = F^2_{n_1} \otimes \cdots \otimes F^2_{n_k}.
\]
Clearly, if we denote by
\[
\C\langle \bm{Z} \rangle_{\n}:= \C\langle Z_1, \ldots, Z_{n_1} \rangle \otimes_{\C} \cdots \otimes_{\C} \C\langle Z_1, \ldots, Z_{n_k} \rangle,
\]
the tensor product of free algebras over $\C$, then $F^2_{\n}$ is naturally a $\C\langle \bm{Z}\rangle_{\n}$-Hilbert module corresponding to $\BS = (\BS_{n_1}, \ldots, \BS_{n_k})$. From this point of view, a closed subspace $\clm \subseteq F^2_{\n}$ is said to be a submodule if
\[
X \clm \subseteq \clm,
\]
for all $X \in \BS_{n_i}$, $i = 1, \ldots, k$. Now if we set
\[
\cle_{\n} = F^2_{n_2} \otimes \cdots \otimes F^2_{n_k},
\]
then Theorem \ref{first-BLHT} (applied to $\cle_{\n}$ in place of $\cle_*$) directly leads to the following corollary concerning representations of submodules of $F^2_{\n}$:

\begin{corollary}\label{second-BLHT}
Let $\clm$ be a closed subspace of the Fock $\n$-module $F^2_{\n}$. Let
\[
\cle = \clm \ominus \sum_{j=1}^{n_1} \BS_{1j} \clm,
\]
and $\cle_{\n} = F^2_{n_2} \otimes \cdots \otimes F^2_{n_k}$. Then $\clm$ is a submodule of $F^2_{\n}$ if and only if there exist an inner multi-analytic operator $\Theta \in R^{\infty}_{n_1} \overline \otimes \clb(\cle, \cle_{\n})$ and multi-analytic operators $\Phi_{ij} \in R_{n_1}^{\infty} \overline \otimes \clb(\cle)$, $i=2, \ldots, k$, and $j = 1, \ldots, n_i$, such that
\beqn
\mathcal M = \Theta (F^2_{n_1}\otimes \mathcal E),
\eeqn
and
\[
\BS_{ij} \Theta = \Theta \Phi_{ij}.
\]
In this case, the Fourier coefficients of $\Phi_{ij}$ are given by
\[
\vp_{ij,\alpha^{t}} = P_{\cle} (S^{\alpha*}\otimes I_{\cle_{\n}}) \BS_{ij}|_{\cle},
\]
for all $\alpha \in F^+_{n_1}$, $i=2, \ldots, k$, and $j = 1, \ldots, n_i$.
\end{corollary}

\begin{remark}\label{remark-13}
A few comments about the above classification result are in order.

\begin{enumerate}
\item In view of our notation $\BS_{n_1} = (\BS_{11}, \ldots, \BS_{1 n_1})$, the Fourier coefficients of $\Phi_{ij}$ can be further simplified to
\[
\vp_{ij,\alpha^{t}} = P_{\cle} \BS_{n_1}^{\alpha*} \BS_{ij}|_{\mathcal E},
\]
for all $\alpha \in F^+_{n_1}$.

\item In view of Remark \ref{phi-is-shift}, it follows that $\Phi_{ij}$ is a pure isometry for each $i=2, \ldots, k$, and $j = 1, \ldots, n_i$.

\item Fix $i \in \{2, \ldots, k\}$. Then by \eqref{eq-thetaphithete}, it follows that
\[
\Phi_{ij} = \Theta^* \BS_{ij} \Theta,
\]
for all $j=1, \ldots, n_i$. Consequently
\[
\Phi_{ip}^* \Phi_{iq} = \delta_{pq} I_{F^2_{n_1} \otimes \cle},
\]
for all $p, q = 1, \ldots, n_i$. Hence, the $n_i$-tuple
\[
\BP_{n_i} = (\Phi_{i1}, \ldots, \Phi_{in_i}),
\]
is a pure isometric tuple on $F^2_{n_1} \otimes \cle$ for all $i=2, \ldots, k$.

\item In the setting of Corollary \ref{second-BLHT}, if $\clm$ is a submodule of $F^2_{\n}$, then
\[
(\BS_{n_1}|_{\clm}, \BS_{n_2}|_{\clm}, \ldots, \BS_{n_k}|_{\clm}) \; \mbox{on} \; \clm,
\]
and
\[
(\BS_{n_1} \otimes I_{\cle}, \BP_{n_2}, \ldots, \BP_{n_k}) \; \mbox{on} \; F^2_{n_1} \otimes \cle,
\]
     are unitarily equivalent.
\end{enumerate}
\end{remark}

We conclude this section with the uniqueness of $(\BP_{n_2}, \ldots, \BP_{n_k})$ on $F^2_{n_1} \otimes \cle$:

\begin{theorem}
In the setting of Corollary \ref{second-BLHT}, let $\clm = \tilde{\Theta} (F^2_{n_1} \otimes \tilde{\cle})$ for some Hilbert space $\tilde{\cle}$ and an inner multi-analytic map $\tilde{\Theta} \in R_{n_1}^\infty \overline{\otimes}\clb(\tilde{\cle}, \cle_{\n})$, and let
\beqn
\BS_{ij} \tilde{\Theta} = \tilde{\Theta} \tilde{\Phi}_{ij},
\eeqn
for some pure isometry $\tilde{\Phi}_{ij} \in R_{n}^{\infty} \overline \otimes \clb(\tilde{\cle})$ and $i=2, \ldots, k$ and $j=1, \ldots, n_i$. Then there exists a unitary $\tau : \mathcal{E} \rightarrow \tilde{\mathcal E}$ such that
\[
\tilde{\Theta}(I_{F^2_{n_1}} \otimes \tau) = \Theta,
\]
and
\[
(I_{F^2_{n_1}} \otimes \tau) \Phi_{ij} =  \tilde{\Phi}_{ij} (I_{F^2_{n_1}} \otimes \tau),
\]
for all $i = 2, \ldots, k$, and $j = 1, \ldots, n_i$.
\end{theorem}
\begin{proof}
By Theorem \ref{BLHT}, there exists a unitary $\tau \in
\clb(\cle,\tilde{\cle})$ such that $\tilde{\Theta}(I_{F^2_{n_1}} \otimes \tau) = \Theta$. Moreover
\[
\tilde{\Theta} (I_{F^2_{n_1}} \otimes \tau) {\Phi}_{ij} = \Theta {\Phi}_{ij} = \BS_{ij} \Theta = \BS_{ij} \tilde{\Theta} (I_{F^2_{n_1}} \otimes \tau),
\]
that is
\[
\tilde{\Theta} (I_{F^2_{n_1}} \otimes \tau) {\Phi}_{ij} = \tilde{\Theta} \tilde{\Phi}_{ij} (I_{F^2_{n_1}} \otimes \tau).
\]
The result now follows from the fact that $\tilde{\Theta}$ is an isometry.
\end{proof}

In particular, if $(n_1, \ldots,n_k) = (1,\ldots,1)$, then the Fock module $F^2_{\n}$ is the Hardy module $H^2(\D^k)$, and hence Corollary \ref{second-BLHT} (and the uniqueness theorem above) recovers \cite[Theorem 3.2]{MMSS}. However, it is worth mentioning that the representation of submodules in Corollary \ref{second-BLHT} somewhat finer than \cite[Theorem 3.2]{MMSS}. The major difference here is the coordinate free approach to submodules of $F^2_{\n}$ as in Theorem \ref{first-BLHT} (for instance, the formalism of $\kappa_i$ in \cite[Theorem 3.1]{MMSS} is not essential in the present consideration). This slightly different technical advantage may actually result in the further development of multivariable operator theory in noncommutative polydomains.

\section{Noncommutative varieties and submodules}\label{section-variety}

First, we briefly recall the necessary definitions and results about noncommutative varieties in $\clb(\clh)^n$ from \cite{Po06}. Given a weak operator topology closed two-sided proper ideal $J$ of $F_n^\infty$, the {\it non-commutative variety} $\mathcal{V}_{J}(\mathcal H)$ corresponding to $J$ is defined by
\beqn
\mathcal{V}_{J}(\clh)  = \{ (X_1, \ldots,X_n) \in {\mf B}^{(n)}( \H): \vp (X_1, \ldots,X_n) = 0 \mbox{ for all } \vp \in   J  \},
\eeqn
where $\vp (X_1, \ldots,X_n)$ is defined in the sense of Popescu's $F_n^\infty$ non-commutative functional calculus for completely non-coisometric contractions \cite{Po95}.

Now let $J$ be a weak operator topology closed two-sided proper ideal in $F^{\infty}_n$. Define
\[
M_J := \overline{JF^2_n} \quad \mbox{and} \quad N_J := F^2_n \ominus JF^2_n.
\]
Since $J$ is a two-sided weakly closed ideal, it follows that $M_J$ is a submodule of $F^2_n$, and hence $N_J$ is a quotient module of $F^2_n$ \cite[Lemma 1.1]{Po06}. Moreover
\[
M_J = \overline{\mbox{span}} \{e_{\alpha} \vp(1) : \vp \in J, \alpha \in F_n^+\} \quad \mbox{ and } \quad N_J = \bigcap_{\vp \in J} \ker \vp^*.
\]
Following \cite{Po06}, define the {\it constrained left (respectively, right) creation operators} as
\[
B_i := P_{N_J} S_i |_{N_J} \quad \mbox{and} \quad W_i := P_{N_J} R_i |_{N_J},
\]
respectively, for all $i=1, \ldots, n$. Therefore
\[
B=(B_1,\ldots,B_n) \quad \mbox{and} \quad W = (W_1, \ldots, W_n),
\]
are $n$-tuples of constrained creation operators on $N_J$. A closed subspace $\clm \subseteq  N_{J} \otimes \clk$, for some Hilbert space $\clk$, is called a \textit{submodule} if $(B_i \otimes I_{\clk}) \clm \subseteq \clm$ for all $i=1, \ldots, n$.

The remarkable fact is that $B \in \mathcal{V}_{J}(N_{J})$ and this constrained tuple $B$ plays the role of model tuple for tuples of operators in noncommutative domains \cite{Po06}. Moreover, we have Popescu's Beurling-Lax-Halmos type theorem for submodules of $N_{J} \otimes \clk$ corresponding to the constrained tuple $(B_1 \otimes I_{\clk}, \ldots, B_n \otimes I_{\clk})$ on $N_J \otimes \clk$ for Hilbert spaces $\clk$:

\begin{theorem}\label{popescu-proof}
\cite[Popescu, Theorem 1.2]{Po06} Let $J \subsetneq F^{\infty}_n$ be a weakly closed two-sided ideal, and let $\mathcal K$ be a Hilbert space. A closed subspace $\clm \subseteq N_{J} \otimes \clk$ is a submodule if and only if there exist a
Hilbert space $\mathcal G$ and a partial isometry
\[
\Theta(W_1,\ldots,W_n) \in \mathscr W (W_1,\ldots,W_n)\overline \otimes B(\mathcal G, \mathcal K)
\]
such that
\[
\mathcal M = \Theta(W_1,\ldots ,W_n) (N_J \otimes \mathcal G).
\]
\end{theorem}

Recall here that $\mathscr W({B}_{1},\ldots,{ B}_{n})$ is the $w^*$-closed (or, weak operator topology closed, as they coincide in this particular situation) algebra generated by $\{I_{N_J}, {B}_{1},\ldots,{ B}_{n}\}$, and (see Popescu \cite[page 396]{Po06} and also see Arias and Popescu \cite{ArPo00})
\[
\mathscr W({B}_{1},\ldots,{ B}_{n}) = P_{N_J} F^\infty_n|_{N_J} = \{\vp(B_1, \ldots, B_n): \vp(S_1, \ldots, S_n) \in F_n^\infty\},
\]
and
\[
\mathscr W({B}_{1},\ldots,{ B}_{n})' = \mathscr  W({W}_{1},\ldots,{ W}_{n}) \quad \mbox{and} \quad \mathscr  W({W}_{1},\ldots,{ W}_{n})' = \mathscr W({B}_{1},\ldots,{ B}_{n}).
\]
Moreover, the noncommutative version of intertwiner lifting \cite{Po89-1} implies that
\begin{equation}\label{eq-W R n infty}
\mathscr  W({W}_{1},\ldots,{W}_{n}) \overline{\otimes} \mathcal B(\cle,\cle_*) = P_{N_{J} \otimes \cle} [R^{\infty}_{n} \overline{\otimes} \mathcal B(\cle, \cle_*)]|_{N_{J} \otimes \cle_*},
\end{equation}
for Hilbert spaces $\cle$ and $\cle_*$. A similar statement holds for $\mathscr  W({B}_{1},\ldots,{B}_{n}) \overline{\otimes} \mathcal B(\cle,\cle_*)$. The elements of $\mathscr W({B}_{1},\ldots,{ B}_{n})$ and $\mathscr  W({W}_{1},\ldots,{ W}_{n})$ are called \textit{constrained multi-analytic operators}. We will use the symbol $\Theta$ (or $\Theta(B_1, \ldots, B_n)$ and $\Theta(W_1,\ldots, W_n)$ if the presence of $(B_1, \ldots, B_n)$ and $(W_1,\ldots, W_n)$, respectively, is not clear from the context) to denote the constrained multi-analytic operators in $\mathscr{W}(B_1, \ldots, B_n)$ and $\mathscr{W}(W_1, \ldots, W_n)$.

The following result furnishes an analogue of Theorem \ref{first-BLHT} in the setting of noncommutative varieties.

\begin{theorem}\label{BLHT-variety}
Let $\cle_*$ be a Hilbert space, $T \in \clb(N_J \otimes \cle_*)$, and let $\clm$ be a closed subspace of $N_J \otimes \cle_*$. Suppose $T (B_i \otimes I_{\cle_*}) = (B_i \otimes I_{\cle_*}) T$ for all $i=1, \ldots n$. The following statements are equivalent:
 \begin{itemize}
\item[(a)] $\clm$ is a submodule of $N_J \otimes \cle_*$ and $T \clm
\subseteq \clm$.
\item[(b)] There exist a closed subspace $\cle$ of $N_J \otimes \cle_*$, a
constrained multi-analytic partial isometry
\[
\Theta(W_1,\ldots,W_n) \in \mathscr W (W_1,\ldots,W_n)\overline \otimes \clb(\cle, \cle_*),
\]
and a constrained multi-analytic operator $\Phi(W_1,\ldots,W_n) \in \mathscr W(W_1,\ldots,W_n)\overline \otimes \clb(\cle)$ such that
\[
\clm = \Theta(W_1,\ldots,W_n) (N_J \otimes \cle),
\]
and
\[
T \Theta(W_1,\ldots,W_n) = \Theta(W_1,\ldots,W_n) \Phi(W_1,\ldots,W_n).
\]
\end{itemize}
\end{theorem}
\begin{proof}
Again, $(b) \Rightarrow (a)$ follows from Douglas' range-inclusion theorem. We now prove that $(a) \Rightarrow (b)$. The idea is to apply Theorem \ref{first-BLHT} along with Popescu's non-commutative version of the commutant lifting theorem in an appropriate sense. Suppose $\clm$ is a submodule of $N_J \otimes \cle_*$ and $T \clm \subseteq \clm$. By the noncommutative commutant lifting theorem (see \cite[Theorem 3.2]{Po89-1}), there exists $\Psi \in R^{\infty}_n\overline{\otimes} \clb(\cle_*)$ such that
\beq\label{lifting-of-t}
\Psi^*|_{N_{_J}\otimes \cle_*} = T^*.
\eeq
Observe that
\[
(F^2_n \otimes \cle_*) \ominus (N_J \otimes \cle_*) = M_J \otimes \cle_*,
\]
is a submodule of $F^2_n \otimes \cle_*$. Define
\[
\clm_J = (M_J \otimes \cle_*) \oplus \clm.
\]
Clearly, $\clm_J \subseteq F^2_n \otimes \cle_*$. First, we claim that $\clm_J$ is $\Psi$-invariant. Indeed, on the one hand
\[
\Psi (M_J \otimes \cle_*) \subseteq (M_J \otimes \cle_*),
\]
as $\Psi^* (N_J \otimes \cle_*) \subseteq (N_J \otimes \cle_*)$ and, on the other hand, since $\clm \subseteq N_J \otimes \cle_*$, we have
\[
\begin{split}
\Psi P_{\clm} & = (P_{M_J \otimes \cle_*} + P_{N_J \otimes \cle_*}) \Psi P_{N_J \otimes \cle_*} P_{\clm}
\\
& = \Big(P_{M_J \otimes \cle_*} \Psi P_{N_J \otimes \cle_*} + P_{N_J \otimes \cle_*} \Psi P_{N_J \otimes \cle_*} \Big) P_{\clm}
\\
& = \Big(P_{M_J \otimes \cle_*} \Psi P_{N_J \otimes \cle_*} + T \Big) P_{\clm},
\end{split}
\]
and hence
\[
\Psi \clm \subseteq \clm_J.
\]
Next we claim that $\clm_J$ is a submodule of $F^2_n \otimes \cle_*$. The proof is similar to the proof of the above claim. Otherwise, one may argue, as in the first paragraph in \cite[page 397]{Po06}, that
\[
\clm_J = \Big(F^2_n \otimes \cle_*\Big) \ominus \Big((N_J \otimes \cle_*) \ominus \clm\Big),
\]
and since $(N_J \otimes \cle_*) \ominus \clm$ is invariant under $(B \otimes I_{\cle_*})^*$, it follows that it is also invariant under $(S \otimes I_{\cle_*})^*$, and hence $\clm_J$ is a submodule of $F^2_n \otimes \cle_*$.

\noindent On the other hand, for each $i=1, \ldots, n$, we have
\[
\begin{split}
P_{M_J \otimes \cle_*} (S_i \otimes I_{\cle_*})^*|_{\clm_J} & = P_{M_J \otimes \cle_*} (S_i \otimes I_{\cle_*})^*|_{M_J \otimes \cle_*} + P_{M_J \otimes \cle_*} (S_i \otimes I_{\cle_*})^*|_{\clm}
\\
& = P_{M_J \otimes \cle_*} (S_i \otimes I_{\cle_*})^*|_{M_J \otimes \cle_*},
\end{split}
\]
as
\[
(S_i \otimes I_{\cle_*})^* \clm \subseteq (S_i \otimes I_{\cle_*})^* (N_J \otimes \cle_*) \subseteq N_J \otimes \cle_*.
\]
This implies
\[
P_{\clm_J} (S_i \otimes I_{\cle_*})^*|_{\clm_J} = P_{M_J \otimes \cle_*} (S_i \otimes I_{\cle_*})^*|_{M_J \otimes \cle_*} + P_{\clm} (S_i \otimes I_{\cle_*})^*|_{\clm_J},
\]
and hence
\begin{equation}\label{eq-15.1}
\tilde{\cle} = \clw \oplus \mathop{\cap}_{i=1}^n \ker \Big(P_{\clm} (S_i \otimes I_{\cle_*})^*|_{\clm_J}\Big),
\end{equation}
where
\begin{equation}\label{eq-15.2}
\tilde{\cle} = \mathop{\cap}_{i=1}^n \ker \Big(P_{\clm_J} (S_i \otimes I_{\cle_*})^*|_{\clm_J}\Big),
\end{equation}
and
\begin{equation}\label{eq-15.3}
\clw = \mathop{\cap}_{i=1}^n \ker \Big(P_{M_J \otimes \cle_*} (S_i \otimes I_{\cle_*})^*|_{M_J \otimes \cle_*}\Big),
\end{equation}
are wandering subspaces of the tuples $(S \otimes I_{\cle_*})|_{\clm_J}$ and $(S \otimes I_{\cle_*})|_{M_J \otimes \cle_*}$ on $\clm_J$ and $M_J \otimes \cle_*$, respectively. We conclude therefore, by Theorem \ref{first-BLHT} along with the proof of Theorem \ref{BLHT} (and in particular, the construction of wandering subspace in \eqref{eq-ker SI}), that there exist an inner multi-analytic operator $\tilde{\Theta} \in R^{\infty}_n \overline{\otimes} \clb (\tilde{\cle}, \cle_*)$ and a multi-analytic operator $\tilde{\Phi} \in R^{\infty}_n\overline{\otimes} \clb(\tilde{\cle})$ such that
\[
\clm_J = \tilde{\Theta} (F^2_n \otimes \tilde{\cle}),
\]
and
\begin{equation}\label{eq-14.0}
M_J \otimes \cle_* = \tilde{\Theta}|_{F^2_n \otimes \clw} (F^2_n \otimes \clw),
\end{equation}
and
\[
\Psi \tilde{\Theta} = \tilde{\Theta} \tilde{\Phi},
\]
and the Fourier coefficients of $\tilde \Phi$ are given by
\begin{equation}\label{eq-14.1}
\tilde \vp_{\alpha^{t}} = P_{\tilde{\cle}} (S^{\alpha*}\otimes I_{\cle_*})\Psi|_{\tilde{\cle}},
\end{equation}
for all $\alpha \in F^+_n$. Now we set
\begin{equation}\label{eq-16.1}
\cle = \tilde{\cle} \ominus  \clw.
\end{equation}
Then
\[
\cle = \mathop{\cap}_{i=1}^n \ker \Big(P_{\clm} (S_i \otimes I_{\cle_*})^*|_{\clm_J}\Big) \subseteq \clm_J \subseteq F^2_n \otimes \cle_*.
\]
Next we define
\[
\Theta \in \mathscr W (W_1,\ldots,W_n)\overline \otimes B(\cle, \cle_*) \quad \mbox{and} \quad \Phi \in \mathscr W(W_1,\ldots,W_n)\otimes \mathcal  B(\mathcal E),
\]
by
\[
\Theta = P_{N_{J} \otimes \cle_*} \tilde{\Theta}|_{N_{J} \otimes \cle},
\]
and
\begin{equation}\label{def-Phi tilde Phi}
\Phi = P_{N_{J} \otimes \cle} \tilde{\Phi}|_{N_{J} \otimes \cle},
\end{equation}
respectively. Now taking into account of the fact that $N_J \otimes \cle_*$ is jointly co-invariant under $(R \otimes I_{\cle_*})$ \cite[Lemma 1.1]{Po06} we have
\[
\tilde{\Theta} ^* (N_{J}\otimes \cle_*) \subseteq N_{J} \otimes \cle,
\]
or equivalently
\begin{equation}\label{eq-17.1}
\tilde{\Theta}^* P_{N_{J}\otimes \cle_*}  = P_{N_{J} \otimes \cle} \tilde{\Theta}^* P_{N_{J}\otimes \cle_*}.
\end{equation}
Note also that
\[
\Psi^*(N_J \otimes \cle_*) \subseteq N_J \otimes \cle_*,
\]
as $\Psi^*|_{N_J \otimes \cle_*} = T^* \in \clb(N_J \otimes \cle_*)$. Using these observations and the fact that $\tilde{\Theta} \tilde{\Phi}=  \Psi \tilde{\Theta}$ we compute
\[
\begin{split}
P_{N_{J} \otimes \cle} \tilde{\Phi}^* \tilde{\Theta}^*|_{N_{J} \otimes \cle_*} & = P_{N_{J} \otimes \cle} \tilde{\Theta}^* \Psi^*|_{N_{J} \otimes \cle_*}
\\
& = P_{N_{J} \otimes \cle} \tilde{\Theta}^*|_{N_{J} \otimes \cle_*} \Psi^*|_{N_{J} \otimes \cle_*}
\\
& = \Theta^* T^*,
\end{split}
\]
which implies
\[
\begin{split}
T \Theta & = P_{N_{J} \otimes \cle_*} \tilde{\Theta} \tilde{\Phi}|_{N_{J} \otimes \cle}
\\
& = P_{N_{J} \otimes \cle_*} \tilde{\Theta} I_{F^2_n \otimes \tilde{\cle}} \tilde{\Phi}|_{N_{J} \otimes \cle}
\\
& = P_{N_{J} \otimes \cle_*} \tilde{\Theta} (P_{F^2_n \otimes \cle} + P_{F^2_n \otimes \clw}) \tilde{\Phi}|_{N_{J} \otimes \cle}.
\end{split}
\]
By \eqref{eq-14.0}, we must have
\[
\tilde{\Theta}(F^2_n\otimes \clw) = M_J \otimes \cle_* \perp N_J \otimes \cle_*,
\]
and hence
\[
P_{N_{J} \otimes \cle_*} \tilde{\Theta} P_{F^2_n \otimes \clw} = 0.
\]
We obtain
\[
T \Theta = P_{N_{J} \otimes \cle_*} \tilde{\Theta} P_{F^2_n \otimes \cle} \tilde{\Phi}|_{N_{J} \otimes \cle}
\]
By \eqref{eq-17.1}, the later implies that
\[
\begin{split}
T \Theta & = P_{N_{J} \otimes \cle_*} \tilde{\Theta} P_{N_{J} \otimes \cle} P_{F^2_n \otimes \cle} \tilde{\Phi}|_{N_{J} \otimes \cle}
\\
& = P_{N_{J} \otimes \cle_*} \tilde{\Theta} P_{N_{J} \otimes \cle} \tilde{\Phi}|_{N_{J} \otimes \cle},
\end{split}
\]
and hence $T \Theta = \Theta \Phi$. Finally, we use \eqref{eq-17.1} again to obtain
\[
\Theta \Theta^* = P_{N_{J} \otimes \cle} \tilde{\Theta} \tilde{\Theta}^*|_{N_{J} \otimes \cle} = P_{N_{J} \otimes \cle} P_{\clm_J}|_{N_{J} \otimes \cle} = P_{\clm}.
\]
Therefore, $\Theta$ is a partial isometry and $\clm = \mbox{ran}\Theta$. We postpone the proof of the fact that $\cle \subseteq N_J \otimes \cle_*$ till Lemma \ref{lemma E subset NJ E}. This completes the proof of the theorem.
\end{proof}

The submodule part of the above theorem reproves Popescu's version of constrained Beurling, Lax and Halmos theorem, namely Theorem \ref{popescu-proof} (or see \cite[Popescu, Theorem 1.2]{Po06}). Here, however, the present proof brings out more geometric flavour (like the fact that $\cle \subseteq N_J \otimes \cle_*$). This will be more evident in the following section.

Moreover, we will return to the constructions of $\tilde{\cle}$ and $\clw$ as in equations \eqref{eq-15.1}, \eqref{eq-15.2} and \eqref{eq-15.3} and the decomposition $\tilde{\cle} = \cle \oplus \clw$ as in \eqref{eq-16.1} in Section \ref{section-DA modules} when we discuss in more detail about the representation of the constrained multi-analytic map $\Phi$.

Now we introduce constrained Fock $\n$-modules (quotient modules of Fock $\n$-modules).

\begin{definition}
Given $\n = (n_1, \ldots, n_k) \in \mathbb{N}^k$ and weak operator topology closed two-sided proper ideal $J_i$ in $F_{n_i}^{\infty}$, $i = 1, \ldots, k$, the constrained Fock $\n$-module $N_{\J_{\n}}$ is defined by
\[
N_{\J_{\n}}:= N_{J_1} \otimes \cdots \otimes N_{J_k}.
\]
\end{definition}

Since $N_{\J_{\n}} \subseteq F^2_{\n}$, the following operators on $N_{\J_{\n}}$ are well defined:
\[
\BB_{ij} = P_{N_{\J_{\n}}} \BS_{ij}|_{N_{\J_{\n}}} \quad \mbox{and} \quad \BW_{ij} = P_{N_{\J_{\n}}} \BR_{ij}|_{N_{\J_{\n}}},
\]
for all $i=1, \ldots, k$ and $j=1, \ldots, n_i$. Set $\BB_{n_i} = (\BB_{i1}, \ldots, \BB_{i n_i})$, the $n_i$-tuple on $N_{J_{\n}}$, and let $\BB = (\BB_{n_1}, \ldots, \BB_{n_k})$. Similarly, define $\BW_{n_i}$ and $\BW$ on $N_{\J_{\n}}$. Observe that
\[
{\bf B}_{i,j} = I_{N_{J_1}} \otimes \cdots \otimes B_{ij}\otimes\cdots\otimes I_{N_{J_k}},
\] and
\[
{\bf W}_{ij} = I_{N_{J_1}} \otimes \cdots \otimes  W_{i,j}\otimes\cdots\otimes I_{N_{J_k}},
\]
for all $i=1, \ldots, k$ and $j=1, \ldots, n_i$. Moreover, if $1 \leq  p < q \leq k$ and $X \in \BB_p$ and $Y \in \BB_q$, then $X Y^* = Y^* X$, that is, $\BB_p$ \textit{doubly commutes} with $\BB_{q}$.

Clearly, $N_{\J_{\n}}$ is a quotient module of the $\C\langle \bm{Z} \rangle_{\n}$-Hilbert module $F^2_{\n}$. From this point of view, a closed subspace $\clm \subseteq N_{\J_{\n}}$ is said to be a \textit{submodule} if
\[
X \clm \subseteq \clm,
\]
for all $X \in \BB_{n_i}$, $i=1, \ldots, k$. The proof of the following corollary concerning submodules of $N_{\J_{\n}}$ is now similar to that of Corollary \ref{second-BLHT}.

\begin{corollary}\label{constrained-submodule}
Let $J_i$ be a weak operator topology closed two-sided proper ideal in $F^\infty_{n_i}$, $i=1, \ldots, k$, and let $\clm$ be a closed subspace of the constrained Fock $\n$-module $N_{\J_{\n}} = N_{J_1} \otimes \cdots \otimes N_{J_k}$. Suppose
\[
\cle_{\n} = N_{J_2} \otimes \cdots \otimes N_{J_k}.
\]
Then $\clm$ is a submodule of the quotient module $N_{\J_{\n}}$ if and only if there exist a Hilbert space $\cle$, a constrained multi-analytic partial isometry $\Theta \in \mathscr W({\bf W}_{11},\ldots,{\bf W}_{1n_1}) \overline \otimes \clb(\cle, \cle_{\n})$, and a constrained multi-analytic operator $\Phi_{ij} \in \mathscr W ({\bf W}_{11},\ldots,{\bf W}_{1 n_1}) \overline{\otimes} \clb(\cle)$ such that
\beqn
\mathcal M = \Theta (N_{J_1} \otimes \mathcal E)
\eeqn
and
\beqn
{\bf B}_{ij} \Theta = \Theta \Phi_{ij},
\eeqn
for all $i=2, \ldots, k$ and  $j=1, \ldots, n_i$.
\end{corollary}

Note, by the way, that the constrained multi-analytic operators
\[
\Phi(W_1, \ldots, W_n) \in \mathscr W(W_1,\ldots,W_n)\overline \otimes \clb(\cle),
\]
in Theorem \ref{BLHT-variety} and
\[
\Phi_{ij} \in \mathscr W ({\bf W}_{11},\ldots,{\bf W}_{1 n_1}) \overline{\otimes} B(\cle),
\]
in Corollary \ref{constrained-submodule} are not canonical. This inconvenience is caused by the fact that $\Psi$ in the proof of Theorem \ref{BLHT-variety} is a lifting of the map $T$, and hence the choice of $\Phi$ is not uniquely determined, in general, by $T$. For now we will leave them alone and take up this issue again in the next section.

\section{Drury Arveson $\n$-modules and the map $\Phi$}\label{section-DA modules}

In this section we continue our discussion of constrained Fock $\n$-modules by looking at special noncommutative varieties. More specifically, here we aim at representing submodules of Drury Arveson $\n$-modules (see the definition of Drury Arveson $\n$-modules below). Moreover, we will analyze the constrained multi-analytic map $\Phi$ of Corollary \ref{constrained-submodule} and see that the representations of $\Phi$ for submodules of Drury Arveson $\n$-modules are more concrete and informative.

We first recall the construction of the Drury-Arveson module. Consider the Fock module $F^2_n$ and let $J$ denote the weakly closed two sided ideal generated by
\begin{equation}\label{eq-ideal J}
\{S_p S_q - S_q S_p: p, q = 1, \ldots, n\} \subseteq F_n^\infty.
\end{equation}
Then the quotient space $N_{J}$ is the symmetric Fock space, $P_{N_{J}} S_i|_{N_J} = P_{N_{J}} W_i|_{N_J}$ for all $i=1, \ldots, n$, and $(P_{N_{J}} S_1|_{N_{J}}, \ldots, P_{N_{J}} S_{n}|_{N_J})$ on $N_J$ and the tuple of multiplication operators $(M_{z_1}, \ldots, M_{z_{n}})$ on the Drury-Arveson space $H^2_n$ are unitarlity equivalent (we will often use this fact implicitly). Moreover
\[
P_{N_{J}} F^{\infty}_{n} |_{P_{N_{J}}} = \mathscr M(H^2_{n}),
\]
where $\mathscr M(H^2_{n})$ denotes the set of multipliers of $H^2_n$ (see \cite{Ar} and \cite{Po06}). Recall also that $H^2_n$ is a reproducing kernel Hilbert space corresponding to the kernel function
\[
K(\z, \w) = (1 - \sum_{i=1}^n z_i \bar{w}_i)^{-1} \quad \quad (\z, \w \in \mathbb{B}^n),
\]
where $\mathbb{B}^n$ denotes the open unit ball in $\C^n$. Moreover, $\mathscr M(H^2_{n})$ is a commutative Banach algebra and is given by
\[
\mathscr M(H^2_{n}) = \{\vp \in \mbox{Hol}(\mathbb{B}^n): \vp H^2_n \subseteq H^2_n\}.
\]
We now consider the Fock $\n$-module $F_{\n}^2 = F^2_{n_1} \otimes \cdots \otimes F^2_{n_k}$. Let $J_i \subseteq F^\infty_{n_i}$ denote the weakly closed two sided ideal generated by the commutators of the creation operators on $F^2_{n_i}$, $i=1, \ldots, k$, as in \eqref{eq-ideal J}. Then the corresponding constrained Fock $\n$-module $N_{\J_{\n}}$, also denoted by $H^2_{\n}$, is the tensor product of Drury-Arveson modules $\{H^2_{n_i}\}_{i=1}^k$, that is
\[
H^2_{\n} = H^2_{n_1} \otimes \cdots \otimes H^2_{n_k}.
\]
We call $H^2_{\n}$ the \textit{Drury-Arveson $\n$-module}. Clearly
\[
\BB_{ij} = \BW_{ij},
\]
and, up to unitarily equivalence, they are equal to
\[
\M_{z_{ij}} = I_{H^2_{n_1}} \otimes\cdots \otimes M_{z_{ij}} \otimes\cdots\otimes I_{H^2_{n_k}},
\]
on $H^2_{\n}$, where $i=1, \ldots, k$ and $j=1, \ldots, n_i$. Also, for simplicity of notation let
\[
\M_{n_i} := (\M_{z_{i 1}}, \ldots , \M_{z_{i n_i}}) \quad \quad (i=1, \ldots, k).
\]
Corollary \ref{constrained-submodule}, in the setting of Drury-Arveson $\n$-module, then yields the following:

\begin{corollary}\label{corollary-DA cor 1}
Let $\n=(n_1,\ldots,n_k) \in \mathbb N^k$, and let $\clm$ be a closed subspace of the Drury-Arveson $\n$-module
\[
H^2_{\n} = H^2_{n_1} \otimes H^2_{n_2} \otimes \cdots \otimes  H^2_{n_k}.
\]
Suppose
\[
\cle_{\n} = H^2_{n_2} \otimes \cdots \otimes  H^2_{n_k}.
\]
Then $\clm$ is a submodule of $H^2_{\n}$ if and only if there exist a Hilbert space $\cle$, a partial isometry $\Theta \in \mathscr M(H^2_{n_1}) \overline \otimes \clb(\cle, \cle_{\n})$, and $\Phi_{ij} \in \mathscr M (H^2_{n_1}) \overline{\otimes} \clb(\cle)$ such that
\beqn
\mathcal M = \Theta (H^2_{n_1} \otimes \mathcal E)
\eeqn
and
\beqn
\M_{z_{ij}} \Theta = \Theta \Phi_{ij},
\eeqn
for all $i=2, \ldots, k$ and  $j=1, \ldots, n_i$.
\end{corollary}

As promised, we now return to address the representation of the constrained multi-analytic operator $\Phi(W_1,\ldots,W_n) \in \mathscr W(W_1,\ldots,W_n)\overline \otimes \clb(\cle)$ in Theorem \ref{BLHT-variety}. And, at the end of this section, in Corollary \ref{corollary-DA cor 2}, we will settle the issue of analytic representations of $\Phi_{ij} \in \mathscr M (H^2_{n_1}) \overline{\otimes} \clb(\cle)$. Here we will work with the same orthogonal decomposition
\[
\tilde{\cle} = \cle \oplus \clw,
\]
as in \eqref{eq-16.1} in the proof of Theorem \ref{BLHT-variety}. Here
\[
\tilde{\cle} = \bigcap_{i=1}^n \Big(\ker (P_{\clm_J} (S_i \otimes I_{\cle_*})^*|_{\clm_J})\Big),
\]
and
\[
\cle =  \bigcap_{i=1}^n \Big(\ker (P_{\clm} (S_i \otimes I_{\cle_*})^*|_{\clm_J}) \Big) \quad \mbox{and} \quad \clw = \bigcap_{i=1}^n \Big(\ker (P_{M_J \otimes \cle_*} (S_i \otimes I_{\cle_*})^*|_{M_J \otimes \cle_*})\Big),
\]
are also as in \eqref{eq-15.1}, \eqref{eq-15.2} and \eqref{eq-15.3}. In this setting, we have the following lemma which seems to be of independent interest:

\begin{lemma}\label{lemma E subset NJ E}
$\cle  \subseteq N_{J} \otimes \cle_{*}$.
\end{lemma}
\begin{proof}
In what follows, $\bigvee$ denote the closed linear span in respective spaces. We observe that
\[
\cle = \tilde \cle \ominus \clw = \tilde{\cle} \cap \clw^{\perp}.
\]
Now we compute
\[
\begin{split}
\clw & = \bigcap_{i=1}^n \Big(\ker (P_{M_J \otimes \cle_*} (S_i \otimes I_{\cle_*})^*|_{M_J \otimes \cle_*})\Big)
\\
& = (M_{J} \otimes \cle_{*}) \bigcap \Big(\bigvee
_{i=1}^n \ran ((S_i \otimes I_{\cle_*})|_{M_J \otimes \cle_*})\Big)^{\perp}
\\
& = (N_{J} \otimes \cle_{*})^{\perp} \bigcap \Big(\bigvee
_{i=1}^n \ran ((S_i \otimes I_{\cle_*})|_{M_J \otimes \cle_*})\Big)^{\perp}
\\
& = \Big((N_{J} \otimes \cle_{*}) \bigvee \Big(\bigvee_{i=1}^n \ran ((S_i \otimes I_{\cle_*})|_{M_J
\otimes \cle_*})\Big)\Big)^{\perp},
\end{split}
\]
which implies
\[
\clw^{\perp} = (N_{J} \otimes \cle_{*}) \bigvee \Big(\bigvee_{i=1}^n \ran ((S_i \otimes I_{\cle_*})|_{M_J \otimes \cle_*})\Big).
\]
On the other hand, since $M_J \otimes \cle_*$ is a submodule of $F^2 \otimes \cle_*$ we have
\[
(N_{J} \otimes \cle_{*}) \perp \Big(\bigvee_{i=1}^n (\ran (S_i \otimes I_{\cle_*})|_{M_J
\otimes \cle_*})\Big),
\]
and hence
\[
\clw^{\perp} = (N_{J} \otimes \cle_{*}) \bigoplus \Big(\bigvee_{i=1}^n \ran ((S_i \otimes I_{\cle_*})|_{M_J \otimes \cle_*})\Big).
\]
Next we simplify $\tilde{\cle}$. Observe that
\[
\tilde{\cle} = \bigcap_{i=1}^n \Big(\ker (P_{\clm_J} (S_i \otimes I_{\cle_*})^*|_{\clm_J})\Big) = \clm_{J} \bigcap \Big(\bigvee_{i=1}^n \ran((S_i \otimes I_{\cle_*})|_{\clm_J})\Big)^{\perp}.
\]
In particular, $\tilde{\cle} \subseteq \clm_{J}$ and
\[
\tilde{\cle} \perp \Big(\bigvee_{i=1}^n \ran((S_i \otimes I_{\cle_*})|_{\clm_J})\Big).
\]
Also since $ M_{J} \otimes \cle_{*} \subseteq \clm_J$ we have
\[
\Big(\bigvee_{i=1}^n \ran((S_i \otimes I_{\cle_*})|_{M_J \otimes \cle_*})\Big) \subseteq \Big(\bigvee_{i=1}^n \ran((S_i \otimes I_{\cle_*})|_{\clm_J})\Big),
\]
and hence
\[
\Big(\bigvee_{i=1}^n \ran((S_i \otimes I_{\cle_*})|_{\clm_J})\Big)^{\perp} \subseteq \Big(\bigvee_{i=1}^n \ran((S_i \otimes I_{\cle_*})|_{M_J \otimes \cle_*})\Big)^{\perp}.
\]
This immediately leads to
\[
\tilde{\cle} \perp \Big(\bigvee_{i=1}^n \ran((S_i \otimes I_{\cle_*})|_{M_J \otimes \cle_*})\Big),
\]
and hence, finally
\[
\cle = \tilde{\cle} \cap \clw^{\perp} = \tilde{\cle} \bigcap \Big((N_{J} \otimes \cle_{*}) \bigoplus \Big(\bigvee_{i=1}^n \ran ((S_i \otimes I_{\cle_*})|_{M_J \otimes \cle_*})\Big)\Big) = \tilde{\cle} \bigcap (N_{J} \otimes \cle_*),
\]
that is $\cle \subseteq N_{J} \otimes \cle_*$. This completes the proof.
\end{proof}

Recall from \eqref{eq-W R n infty} that if $\Phi(W_1,\ldots,W_n) \in \mathscr W(W_1,\ldots,W_n)\overline \otimes \clb(\cle)$ is a constrained multi-analytic operator, then there exists (not necessarily unique)
\[
\tilde{\Phi}(R_1, \ldots, R_n) \in R_n^\infty \overline \otimes \clb(\cle),
\]
such that
\[
\Phi = P_{N_J \otimes \cle} \tilde{\Phi}|_{N_J \otimes \cle}.
\]
Evidently being a solution of the commutant lifting, the multi-analytic operator $\tilde{\Phi}$ is not unique and hence any possible definition of Fourier coefficients of $\Phi$ will be ambiguous (for instance, see \eqref{eq-general vp} below). However, as we shall see soon, for constrained Fock $\n$-modules the situation is somewhat favourable.

First we turn to constrained multi-analytic operator $\Phi$ in Theorem \ref{BLHT-variety}. Here 
\[
\Phi = P_{N_J \otimes \cle} \tilde{\Phi}|_{N_J \otimes \cle},
\]
and $\tilde{\Phi} \in R_n^\infty \overline{\otimes} \clb(\cle)$ (see \eqref{def-Phi tilde Phi}). We note that by Theorem \ref{first-BLHT}, the Fourier coefficients of $\tilde{\Phi}$, as constructed in the proof of Theorem \ref{BLHT-variety} (and also see \eqref{eq-14.1}) are given by
\[
\tilde \vp_{\alpha^{t}} = P_{\tilde{\cle}} (S \otimes I_{\cle_*})^{\alpha*} \Psi|_{\tilde{\cle}} \quad \quad (\alpha \in F^+_{n}).
\]
In this case, we define the \textit{Fourier coefficients of the constrained multi-analytic operator} $\Phi \in \mathscr W(W_1,\ldots,W_n)\overline \otimes \clb(\cle)$ corresponding to $\tilde{\Phi}$ as
\[
\vp_{\alpha^{t}}:= P_{\cle} \tilde{\phi}_{\alpha^{t}}|_{\cle} =  P_{\cle} (S \otimes I_{\cle_*})^{\alpha*} \Psi|_{\cle} \quad \quad (\alpha \in F^+_{n}).
\]
We proceed now to describe the Fourier coefficients $\vp_{\alpha^{t}}$ in detail as follows. Since, $\Psi^*|_{N_J \otimes \cle_*} = T^*$ by \eqref{lifting-of-t}, and $\cle  \subseteq N_{J} \otimes \cle_{*}$ by Lemma \ref{lemma E subset NJ E}, it follows that
\[
\Psi|_{\cle} = P_{N_J \otimes \cle_*} \Psi|_{\cle} + P_{M_J \otimes \cle_*} \Psi|_{\cle} = T|_{\cle} + P_{M_J \otimes \cle_*} \Psi|_{\cle},
\]
and hence
\begin{equation}\label{eq-general vp}
\vp_{\alpha^t} =  P_{{\cle}}(S \otimes I_{\cle_*})^{\alpha*} T|_{\cle} + P_{\cle} (S \otimes I_{\cle_*})^{\alpha*} P_{M_J \otimes \cle_*}\Psi|_{\cle},
\end{equation}
for all $\alpha \in F^+_{n}$. Here note that the appearance of $P_{M_J \otimes \cle_*}\Psi$ (here $\Psi$ is a lifting of $T$ as in \eqref{lifting-of-t}) is not so convenient.

We obviate the above inconvenience by restricting $T$ to tensor product of operators. So we now move to the setting of Corollary \ref{constrained-submodule}. Here, we treat $N_J$ as $N_{J_1}$, $\cle_*$ as
\[
\cle_{\n} = N_{J_2} \otimes \cdots \otimes N_{J_k},
\]
and $T$ as ${\bf B}_{ij}$ on $N_{J_1} \otimes \cle_{\n}$ for all $i=2, \ldots, k$, and $j=1, \ldots, n_i$. In this case, consequently one may choose a lifting  $\Psi_{ij}$ in $R_{n_1}^\infty \overline{\otimes} \clb(\cle_{\n})$ as the constant multi-analytic operator
\[
Y_{ij} :=I_{F^2_{n_1}} \otimes (I_{N_{J_2}} \otimes \cdots\otimes B_{ij} \otimes \cdots \otimes I_{N_{J_k}}),
\]
for all $i=2, \ldots, k$, and $j=1, \ldots, n_i$. Let $\vp_{ij, \alpha^t}$ be the $\alpha$-th Fourier coefficient of the multi-analytic operator $\Phi_{ij}$ in $\mathscr W(W_1,\ldots,W_n)\overline \otimes \clb(\cle)$, $i=2, \ldots, k$, and $j=1, \ldots, n_i$. Then \eqref{eq-general vp} implies that
\[
\vp_{ij,\alpha^t} = P_{\cle} (S_{n_1} \otimes I_{\cle_{\n}})^{\alpha*} \BB_{ij}|_{\cle} + P_{\cle} (S_{n_1} \otimes I_{\cle_{\n}})^{\alpha*} P_{M_{J}\otimes \cle_*} Y_{ij}|_{\cle}.
\]
Since, $\cle \subseteq N_{J_1} \otimes \cle_{*}$ by Lemma \ref{lemma E subset NJ E}, and $N_{J_1} \otimes \cle_{*}$ is invariant under $Y_{ij}$ by construction, we have that
\[
P_{M_{J}\otimes \cle_*} Y_{ij}|_{\cle} = 0,
\]
and hence
\[
\vp_{ij,\alpha^t} = P_{\cle} (S_{n_1} \otimes I_{\cle_{\n}})^{\alpha*} \BB_{ij}|_{\cle} = P_{\cle} \BB_{n_1}^{\alpha*} \BB_{ij}|_{\cle},
\]
as $\BB_{ij}|_{\cle} = P_{N_{J_1} \otimes \cle_{\n}} \BB_{ij}|_{\cle}$ and $(S_{1q} \otimes I_{\cle_{\n}})^*|_{N_{J_1} \otimes \cle_{\n}} = \BB_{1q}^*$ for all $q=1, \ldots, n_1$. We have thus arrived at the following companion result to Corollary \ref{constrained-submodule}:

\begin{corollary}\label{cor-Fourier coeff of Phi ij}
In the setting of Corollary \ref{constrained-submodule}, for each $\alpha \in F^+_n$, $i=2, \ldots, k$, and $j=1, \ldots, n_i$, the $\alpha$-th Fourier coefficient of the constrained multi-analytic operator $\Phi_{ij}$ is given by
\[
\vp_{ij,\alpha^t} = P_{\cle} \BB_{n_1}^{\alpha*} \BB_{ij}|_{\cle}.
\]
\end{corollary}

Finally, in view of the above corollary, we now turn to analytic representations of $\Phi_{ij} \in \mathscr M (H^2_{n_1}) \overline{\otimes} \clb(\cle)$ in Corollary \ref{corollary-DA cor 1}. So now on, we will be following the setting of Corollary \ref{corollary-DA cor 1}.

Denote by $\sigma$ the symmetrization map from $F_{n_1}^+$ to $\ZP^{n_1}$. Let $i = 2, \ldots, k$ and $j = 1, \ldots, n_i$. Then
\[
\BB_{ij}^{\alpha} = P_{N_{J_1} \otimes \cle_{\n} }(S_{n_1}^{\alpha}\otimes I_{\cle_{\n}})|_{N_{J_1}\otimes \cle_{\n}} = \M_{n_1}^{\sigma(\alpha)},
\]
whence
\[
\BB_{ij}^{\alpha*} = \BW_{ij}^{\alpha} =\M_{n_1}^{\sigma(\alpha)},
\]
for all $\alpha \in F^+_{n_1}$. Then, by Corollary \ref{constrained-submodule} we have
\[
\vp_{ij,\alpha^t} = P_{\cle} (\M_{n_{1}})^{*\m} \M_{z_{ij}}|_{\cle} \quad \quad (\alpha \in F^+_n)
\]
where $\sigma (\alpha) = \m$. Moreover, a standard computation shows, for each $\m \in \ZP^{n_1}$, that
\[
\# \{\alpha \in F^{+}_{n_1} : \sigma(\alpha) = \m\} = \frac{|\m| !}{\m !} := \frac{(\displaystyle\sum_{i=1}^{n_1} m_i)!}{m_1! \cdots m_{n_1}!}.
\]
Then, for each $\z \in \mathbb{B}^{n_1}$ we have
\[
\begin{split}
\Phi_{ij}(\z) & =  \sum_{\m \in \ZP^{n_1}} \frac{|\m| !}{\m !} (P_{\cle} \M_{n_1}^{*\m} \M_{z_{ij}}|_{\cle}) \z^{\m}
\\
&= P_{\cle} \Big( \sum_{t \in \ZP} (\sum_{\overset{\m \in \ZP^{n_1}}{|\m| = t}} \frac{|\m|!}{\m !} \M_{n_1}^{*\m} \M_{z_{ij}} \z^{\m}) \Big)|_{\cle},
\end{split}
\]
and hence
\[
\Phi_{ij}(\z) = P_{\cle} \Big(I - \sum_{m=1}^{n_1} z_m \M_{z_{1m}}^* \Big)^{-1} \M_{z_{ij}}|_{\cle}.
\]
We have proved the following.

\begin{corollary}\label{corollary-DA cor 2}
In the setting of Corollary \ref{corollary-DA cor 1}, for each $i=2, \ldots, k$ and $j=1, \ldots, n_i$, the multiplier $\Phi_{ij} \in \mathscr M (H^2_{n_1}) \overline{\otimes} \clb(\cle)$ can be represented as
\[
\Phi_{ij}(\z) = P_{\cle} \Big(I - \sum_{m=1}^{n_1} z_m \M_{z_{1m}}^* \Big)^{-1} \M_{z_{ij}}|_{\cle} \quad \quad (\z \in \mathbb{B}^{n_1}).
\]
\end{corollary}

\section{An example and concluding remarks}\label{concluding remarks}

Structure of isometries (that is, the von Neumann and Wold decomposition theorem), Beurling, Lax and Halmos theorem, Sarason's commutant lifting theorem, and the Sz.-Nagy and Foias analytic model theory have been inseparable companions in single variable operator theory and function theory. These concepts are increasingly accepted as stepping stones to the development of (both commutative and noncommutative) multivariable operator theory. However, there are a number of interesting and vital results that hold for single bounded linear operators but do not hold in general for commuting and noncommuting $n$-tuples, $n \geq 2$, of operators. Here we aim to present one such example.

First, we recall the noncommutative version of Beurling, Lax and Halmos theorem (see Theorem \ref{BLHT}): Let $\cle_*$ be a Hilbert space and let $\clm$ be a closed subspace of $F^2_n \otimes \cle_*$. Then $\clm$ is a submodule of $F^2_n \otimes \cle_*$ if and only if there exist a Hilbert space $\cle$ and an inner multi-analytic operator $\Theta(R_1, \ldots, R_n) : F^2_n \otimes \cle \rightarrow F^2_n \otimes \cle_*$ such that
\[
\clm = \Theta(R_1, \ldots, R_n) \Big(F^2_n \otimes \cle\Big).
\]
If we assume in addition that $n=1$, then
\[
\mbox{dim} \, \cle \leq \mbox{dim} \, \cle_*.
\]
This inequality plays a crucial role in single variable operator theory. Here, on the contrary, if $n>1$, then we show that such dimension inequality do not hold in general, specifically we construct an inner multi-analytic operator $\Theta(R_1, \ldots, R_n) \in R_n^\infty \overline{\otimes} \clb(\cle, \cle_*)$ such that $\mbox{dim} \, \cle > \mbox{dim} \, \cle_*$. Clearly and necessarily, here one must consider finite dimensional Hilbert spaces $\cle_*$.

\begin{example}
Let $n > 1$ and let $\tilde{\cle}_1, \ldots, \tilde{\cle}_n$ and $\cle_*$ be Hilbert spaces. Suppose
\[
\mbox{dim} \, \tilde{\cle}_i = \mbox{dim} \, \cle_* = m \; (< \infty),
\]
for all $i=1, \ldots, n$, and let
\[
\cle = \bigoplus_{i=1}^n \tilde{\cle}_i.
\]
Let $\{e_{ij}\}_{j=1}^m$ be an orthonormal basis of $\tilde{\cle}_i$, $i=1, \ldots, n$, and let $\{f_j\}_{j=1}^m$ be that of $\cle_*$. Then
\[
\mbox{dim}\; \cle = m n.
\]
Now for each $i=1, \ldots, n$, we define linear operator $\theta_i : \cle \raro \cle_*$ by
\[
\theta_i(e_{pq}) =
\begin{cases} f_q & \mbox{if}~ p=i
\\
0 & \mbox{if}~ p \neq i. \end{cases}
\]
An easy calculation reveals that
\[
\sum_{i=1}^n \theta_i^* \theta_i = I_{\cle}.
\]
Set
\[
\Theta(R_1, \ldots, R_n) = \sum_{i=1}^n R_i \otimes \theta_i.
\]
Clearly
\[
(S_i \otimes I_{\cle_*}) \Theta(R_1, \ldots, R_n) = \Theta(R_1, \ldots, R_n) (S_i \otimes I_{\cle}),
\]
for all $i=1, \ldots, n$, that is, $\Theta(R_1, \ldots, R_n) \in F_n^\infty \overline{\otimes} \clb(\cle, \cle_*)$. Moreover
\[
\Theta(R_1, \ldots, R_n)^* \Theta(R_1, \ldots, R_n) = \sum_{i=1}^n R_i^* R_i \otimes \theta^*_i \theta = I_{F^2_n \otimes \cle},
\]
that is, $\Theta(R_1, \ldots, R_n)$ is an inner multi-analytic operator in $F_n^\infty \overline{\otimes} \clb(\cle, \cle_*)$, where, on the other hand
\[
mn = \mbox{dim} \, \cle > \mbox{dim} \, \cle_* = n.
\]
\end{example}

It is now evident, from the above example point of view, to the least, that extensions of some of the concepts (for instance see the appendix in \cite{MMSS}) of submodules of the Hardy space over polydisc in its full generality is not possible in the context of multivariable (both commutative and noncommutative) operator theory and noncommutative varieties.

All the main results of this paper remain valid if we replace Fock $\n$-modules, constrained Fock $\n$-modules and Drury-Arveson $\n$-modules by the respective vector-valued counterparts and the proofs carry over verbatim.

\vspace{0.1in}

\noindent\textbf{Acknowledgement:}
The research of the second named author is supported by NBHM  (National Board of Higher Mathematics, India) post-doctoral fellowship no: 0204/27-\\/2019/R\&D-II/12966.
The research of the third named author is supported in part by NBHM grant NBHM/R.P.64/2014, and the Mathematical Research Impact Centric Support (MATRICS) grant, File No: MTR/2017/000522 and Core Research Grant, File No: CRG/2019/000908, by the Science and Engineering Research Board (SERB), Department of Science \& Technology (DST), Government of India. The third author also would like to thanks the Institute of Mathematics of the Romanian Academy, Bucharest, Romania, for its warm hospitality during a visit in May 2019.

\end{document}